\documentclass[11pt]{article}

\usepackage[utf8]{inputenc}
\usepackage{amsmath,amsfonts,latexsym, xspace,amsthm,graphicx, amssymb}
\usepackage{enumerate}
\usepackage[letterpaper,left=1in,right=1in,top=1in,bottom=1in]{geometry}
\usepackage{amsmath, tikz,amssymb,color,comment}
\usepackage{needspace}

\newcommand{\Gnp}{G_{n,p}}

\newcommand{\Gncntext}{G_{n,c/n}}
\newcommand{\Gnpdash}{G_{n,p'}}
\newcommand{\Gnpdashdash}{G_{n,p''}}
\newcommand{\Gnm}{G_{n,m}}

\newcommand{\Rab}{R_{n,\alpha,\beta}}
\newcommand{\Rabk}{R_{n,\alpha,\beta,k}}

\usepackage{amsthm,mathrsfs,fancyhdr,subfigure}
\usepackage{graphicx} 
\usepackage[noadjust,sort]{cite}
\newcommand{\E}{{\mathbb E}}
\newcommand{\R}{{\mathbb R}}
\newcommand{\ind}{{\mathbb I}}

\newcommand{\pr}{{\mathbb P}}
\newcommand{\var}{{\rm var}}
\newcommand{\Bin}{{\rm Bin}}

\newcommand{\vol}{{\rm vol}}
\newcommand{\isol}{{\rm isol}}

\renewcommand{\AA}{ \mathcal{A} }
\newcommand{\cB}{\mathcal{B}}
\newcommand{\CC}{\mathcal{C}}
\newcommand{\LL}{\mathcal{L}}
\newcommand{\Z}{{\mathbb Z}}

\newcommand{\cA}{{\mathcal A}}
\newcommand{\cE}{{\mathcal E}}

\newcommand{\GG}{\mathcal{G}}

\newcommand{\sg}{\bar{\lambda}}

\newcommand{\q}{q^*}
\newcommand{\eps}{\varepsilon}

\newcommand{\Abs}[1]{\bigl|#1\bigr|}

\newtheorem{thm}{Theorem}[section]
\newtheorem{defn}{Definition}[section]
\newtheorem{cor}[thm]{Corollary}

\newtheorem{rem}[thm]{Remark}
\newtheorem{conj}[thm]{Conjecture}
\newtheorem{lemma}[thm]{Lemma}
\newtheorem{prop}[thm]{Proposition}

\let\stdcaption\caption
\usepackage{floatpag}
\let\caption\stdcaption

\usepackage{tikz}
\usepackage{tkz-berge}
\usepackage{tkz-graph}
\usepackage{tkz-arith}
\usepackage{tkz-euclide}
\usepackage{tkz-base}
\usepackage{pgf}

\numberwithin{equation}{section}

\newcommand{\cokfourminus}{
  \begin{tikzpicture}[baseline=-0.6ex,scale=0.25]
  \tikzstyle{vertex}=[circle,fill=black, minimum size=2pt,inner sep=1pt]
  \node[vertex] (v1) at (-0.5, 0.5){};
  \node[vertex] (v2) at (-0.5,-0.5){};
  \node[vertex] (v3) at (0.5, 0.5){};
  \node[vertex] (v4) at (0.5,-0.5){};
  \draw (v2)--(v4);
  \end{tikzpicture}}

\newcommand{\cofourcycler}{
  \begin{tikzpicture}[baseline=-0.6ex,scale=0.25]
  \tikzstyle{vertex}=[circle,fill=black, minimum size=2pt,inner sep=1pt]
  \node[vertex] (v1) at (-0.5, 0.5){};
  \node[vertex] (v2) at (-0.5,-0.5){};
  \node[vertex] (v3) at (0.5, 0.5){};
  \node[vertex] (v4) at (0.5,-0.5){};
  \draw (v1)--(v3) (v2)--(v4);
  \end{tikzpicture}}

\newcommand{\threeedgepath}{
  \begin{tikzpicture}[baseline=-0.6ex,scale=0.25]
  \tikzstyle{vertex}=[circle,fill=black, minimum size=2pt,inner sep=1pt]
  \node[vertex] (v1) at (-0.5, 0.5){};
  \node[vertex] (v2) at (-0.5,-0.5){};
  \node[vertex] (v3) at (0.5, 0.5){};
  \node[vertex] (v4) at (0.5,-0.5){};
  \draw (v3)--(v1)--(v2)--(v4);
  \end{tikzpicture}}

\title{Modularity of Erd\H{o}s-R\'enyi random graphs}

\date{\today}

\author{Colin McDiarmid, Fiona Skerman} 

\begin{document}

\maketitle

\begin{abstract}
For a given graph $G$, each partition of the vertices has a modularity score, with higher values indicating that the partition better captures community structure in $G$. The modularity $\q(G)$ of the graph $G$ is defined to be the maximum over all vertex partitions of the modularity score, and satisfies $0\!\leq \!\q(G)\!< \!1$. Modularity is at the heart of the most popular algorithms for community detection.

We investigate the behaviour of the modularity of the Erd\H{o}s-R\'enyi random graph $\Gnp$ with~$n$ vertices and edge-probability $p$. Two key findings are that the modularity is $1+o(1)$ with high probability (whp) for $np$ up to $1+o(1)$ and no further; and when $np \geq 1$ and $p$ is bounded below~1, it has  order $(np)^{-1/2}$ whp, in accord with a conjecture by Reichardt and Bornholdt in~2006.

We also show that the modularity of a graph is robust to changes in a few edges, in contrast to the sensitivity of  optimal vertex partitions.
\end{abstract}

\maketitle

\section{Introduction} \label{sec.intro}

We start this section with some background and definitions, and then present our results on the modularity value $\q$ of the random graph $\Gnp$, followed by corresponding results for the random graph $\Gnm$ with $m$ edges. After that, we sketch previous work on modularity, and then give a plan of the rest of the paper. 

\needspace{6\baselineskip}
\subsection{Definitions}
The large and ever-increasing quantities of network data available in many fields has led to great interest in techniques to discover network structure. We want to be able to identify if a network can be decomposed into dense clusters or `communities'. 

Modularity was introduced by Newman and Girvan in 2004~\cite{NewmanGirvan}. It gives a measure of how well a graph can be divided into communities, and now forms the backbone of the most popular algorithms used to cluster real data~\cite{popular}. Here a `community' is a collection of nodes which are more densely interconnected than one would expect -- see the discussion following the definition of modularity below. There are many applications, including for example protein discovery, and identifying connections between websites: see~\cite{fortunato2010community} and~\cite{porter2009communities} for surveys on the use of modularity for community detection in networks. Its widespread use and empirical success in finding communities in networks makes modularity an important function to understand mathematically. 

Given a graph $G$, we give a modularity score to each vertex partition (or `clustering') : the modularity $\q(G)$ (sometimes called the `maximum modularity') of $G$ is defined to be the maximum of these scores over all vertex partitions. For a set $A$ of vertices, let $e(A)$ be the number of edges within $A$, and let the \emph{volume} $\vol(A)$ be the sum over the vertices $v$ in $A$ of the degree $d_v$. 

\begin{defn}[Newman \& Girvan~\cite{NewmanGirvan},  see also Newman~\cite{NewmanBook}]\label{def.mod}
Let $G$ be a graph with $m\geq 1$ edges. For a partition $\cA$ of the vertices of~$G$, 
the modularity score of $\cA$ on $G$ is 
\begin{equation*} q_\cA(G) = 
\frac{1}{2m}\sum_{A\in \cA} \sum_{u,v \in A}  
\left( {\mathbf 1}_{uv\in E} - \frac{d_u d_v}{2m} \right)
= \frac{1}{m}\sum_{A \in \cA} e(A) - \frac{1}{4m^2}\sum_{A\in \cA} \vol (A)^2; \end{equation*}
and the modularity of $G$ is $\q(G)=\max_\cA q_{\cA}(G)$, where the maximum is over all vertex partitions~$\cA$ of~$G$.
\end{defn}

Isolated vertices are irrelevant. We need to give empty graphs (graphs with no edges) some modularity value.  Conventionally we set $\q(G)=0$ for each such graph $G$ (though the value will not be important). The second equation for $q_{\cA}(G)$ expresses modularity as the difference of two terms, the \emph{edge contribution} or \emph{coverage} $q^E_\cA(G)=\tfrac{1}{m}\sum_A e(A)$, and the \emph{degree tax} $q^D_\cA(G)=\tfrac{1}{4m^2}\sum_A\vol(A)^2$.  
Since $q^E_{\cA}(G) \leq 1$ and $q^D_{\cA}(G) >0$, we have $q_{\cA}(G)<1$ for any non-empty graph $G$.  Also, the trivial partition $\cA_0$ with all vertices in one part has $q^E_{\cA_0}(G) = q^D_{\cA_0}(G) =1$, so $q_{\cA_0}(G)  = 0$.  Thus we have 
\begin{equation*} 0 \leq \q(G) < 1.\end{equation*}
Suppose that we pick uniformly at random a multigraph with degree sequence $(d_1, \ldots, d_n)$ where $\sum_v d_v = 2m$. Then the expected number of edges between distinct vertices $u$ and $v$ is $d_u d_v/(2m-1)$. This is the original rationale for the definition: whilst rewarding the partition for capturing edges within the parts, we should penalise by (approximately) the expected number of such edges. The corresponding modularity scores for bipartite or directed graphs \cite{porter2007community,leicht2008community,guimera2007module} and for hypergraphs~\cite{kaminski2018clustering} have also been defined and require an amended degree tax.
\emph{A differentiation between graphs which are truly modular and those which are not can ...  only be made if we gain an understanding of the intrinsic modularity of random graphs.} -- Reichardt and Bornholdt \cite{trulymodular}. In this paper we investigate the likely value of the modularity of an Erd\H{o}s-R\'enyi random graph. Let~$n$ be a positive integer.  Given $0 \leq p \leq 1$, the random graph $\Gnp$ has vertex set $[n]:=\{1,\ldots,n\}$ and the $\binom{n}{2}$ possible edges appear independently with probability $p$.  Given an integer $m$ with $0 \leq m \leq \binom{n}{2}$, the random graph $\Gnm$ is sampled uniformly from the $m$-edge graphs on vertex set $[n]$.  These two random graphs are closely related when $m \approx \binom{n}{2} p$: we shall focus on~$\Gnp$, but see results on~$\q(\Gnm)$ in Section~\ref{subsec.Gnm}.

For a sequence of events $A_n$ we say that $A_n$ holds \emph{with high probability (whp)} if $\pr(A_n) \to 1$ as $n \to \infty$.
For a sequence of random variables $X_n$ and a real number $a$, we write $X_n \overset{p} \rightarrow a$ if $X_n$ converges in probability to $a$ as $n \to \infty$ (that is, if for each $\eps>0$ we have $|X_n-a| < \eps$ whp). For $x=x(n)$ and $y=y(n)$ we write $x \sim y$ to indicate $x= (1+o(1))y$ as $n \to \infty$.

\needspace{6\baselineskip}
\subsection{Results on the modularity of the random graph $\Gnp$}
\label{subsec.Gnpresults}

Our first theorem, the Three Phases Theorem, gives the big picture.   The three phases correspond to when (a) the expected vertex degree (essentially $np$) is at most about 1, (b) bigger than 1 but bounded, or (c) tending to infinity.

\needspace{8\baselineskip}
\begin{thm}\label{thm.usER}
Let $p=p(n)$ satisfy $0 < p \leq 1$. \begin{enumerate}\vspace{-3mm}
\item[(a)] If $n^2p \rightarrow \infty$ and $np\leq 1+o(1)$ then $\; \q(\Gnp) \overset{p}\rightarrow 1$.

\item[(b)] Given constants $1< c_0 \leq c_1$,  there exists $\delta=\delta(c_0,c_1)>0$ such that if $c_0 \leq np \leq c_1$ for $n$ sufficiently large, then whp $\; \delta<\q(\Gnp)<1-\delta$. 

\item[(c)] If $np\rightarrow \infty$ then $\; \q(\Gnp) \overset{p}\rightarrow 0$.
\end{enumerate}
\end{thm}

Following the above general theorem we now give more detailed results. In each of the above parts (a), (b), (c) of Theorem~\ref{thm.usER} we can be more precise.  Let us start with the sparse case,  for $p$ ranging from $0$ up to a little above $1/n$, corresponding to part (a) and a little into part (b).
In this sparse case, the modularity $\q$ is near 1 whp, so we are interested in the \emph{modularity deficit} $1-\q$. Given a graph $G$, let $\CC$ denote the connected components partition, in which the parts are the vertex sets of the connected components of $G$.
\begin{thm}\label{thm.sparse}
(i) If $n^{3/2}p = o(1)$ then whp exactly one of the following two statements holds: either there are no edges (that is $e(\Gnp)=0$), or $\, \q(\Gnp)= q_{\CC}(\Gnp) = 1- 1/e(\Gnp)$.

(ii) 
If $n^2p \rightarrow \infty$ and 
$np \leq 1- (\log n)^{1/2} n^{-1/4}$, 
then
\begin{equation*} \q(\Gnp) = q_{\CC}(\Gnp) = 1 - \Theta(\tfrac{1}{n^2p(1-np)}) \;\;\; \mbox{ whp}.\end{equation*}

(iii)  
If $0 < \eps \leq 1/4$ and $1 \leq np \leq 1+\eps$ then $\q(\Gnp) \geq q_{\CC}(\Gnp)> 1- (4\eps)^2$ whp.
\end{thm} 

Part (i) above shows that we need the condition $n^2 p \to \infty$ in Theorem~\ref{thm.usER} part (a). For, if $n^2p$ is bounded, say $n^2 p \leq \alpha$ for some $\alpha>0$, then $e(\Gnp) \leq \alpha$ whp ; and so by part (i), whp $\q(\Gnp) \leq 1 - 1/\alpha$ (or there are no edges).  (In fact, for each $m\geq 1$, the maximum value of $\q(G)$ over all $m$-edge graphs $G$ is $1-1/m$, see~\cite{extreme}.) 

In parts (i) and (ii), the connected components partition $\CC$ is the unique optimal partition, ignoring isolated vertices -- see Proposition~\ref{prop.CCbestonly}. The upper bound condition on $np$ in part (ii) is nearly best possible: it is shown in~\cite{thesis} that, if $np \geq 1-  (\log n)^{1/4} n^{-1/4}$ and $np=O(1)$, then whp there is a partition for~$\Gnp$ with strictly higher modularity than the connected components partition. In part (iii), corresponding to the lower bound on $q_{\CC}(\Gnp)$ there is a matching upper bound on $q_{\CC}(\Gnp)$, see Lemma~\ref{lem.qCC};  but we do not have a matching upper bound for $\q(\Gnp)$, see Conjecture~\ref{conj.squaret}.

The next theorem confirms the $(np)^{-1/2}$ growth rate which was conjectured in~\cite{trulymodular} to hold when $np$ is above 1 and not too big: further details of the prediction are given in Section~\ref{subsec.prev}. 

\begin{thm}\label{thm.growthRate}
There exists $b$ such that for all $0< p=p(n) \leq 1$ we have $\q(\Gnp)<\frac{b}{\sqrt{np}}$ whp.
Also, given $0<\eps<1$, there exists $a =a(\eps) >0$ such that, if $p=p(n)$ satisfies $np\geq 1$ and $p \leq 1-\eps$ for $n$ sufficiently large, then  $\q(\Gnp) > \frac{a}{\sqrt{np}}$ whp.
\end{thm}

The lower bound here comes from analysing the algorithm \emph{Swap}: for further information see Section~\ref{subsec.lb} and in particular Theorem~\ref{thm.lowersqrt}.
Theorem~\ref{thm.growthRate} implies part (c) of Theorem~\ref{thm.usER}, and implies part (b) of Theorem~\ref{thm.usER} except for the upper bound $1\!-\!\delta$ on $\q(\Gnp)$ when $np$ is small (in particular, not when $np \leq b^2$). The upper bound in Theorem~\ref{thm.growthRate} implies that modularity values distinguish the stochastic block model (defined in Section~\ref{subsec.sb}) from the Erd\H{o}s-R\'enyi model whp, when the probabilities are just a constant factor past the detectability threshold, as we explain in Remark~\ref{rem.distinguish}. 

\needspace{4\baselineskip}
As an immediate corollary of Theorem~\ref{thm.growthRate} we have:
\begin{cor}\label{cor.growthRate}
There exists $0<a<b$ such that, if $1/n \leq p=p(n) \leq 0.99$ then
\begin{equation*} \frac{a}{\sqrt{np}} <\q(\Gnp)<\frac{b}{\sqrt{np}} \;\;\; \mbox{ whp}. \end{equation*}
\end{cor}

A higher modularity score is taken to indicate a better community division. Thus to determine whether a clustering $\cA$ in a graph $G$ shows significant community structure we should compare $q_\cA(G)$ to the likely (maximum) modularity for an appropriate null model, that is, to the likely value of $\q(\tilde{G})$ for null model $\tilde{G}$.  It is an interesting question which null model may be most appropriate in a given situation.  For example, real networks have been shown to exhibit power law degree behaviour, and so null models which can mimic this have been suggested, for example the Chung-Lu model~\cite{chunglupowerlaw} or random hyperbolic graphs~\cite{hyperbolicIntro}. However, a natural minimum requirement is not to consider a community division of a real network as statistically significant unless its modularity score is higher than the typical modularity value of an Erd\H{o}s-R\'enyi random graph $\Gnp$ of the same edge density.

\needspace{6\baselineskip}
\subsection{Results on the modularity of the random graph $\Gnm$}
\label{subsec.Gnm}

Each of our results on the random graph $\Gnp$ has a counterpart for $\Gnm$, which can be deduced quickly as a corollary. 
%
Corresponding to Theorems~\ref{thm.usER},~\ref{thm.sparse} and~\ref{thm.growthRate} (and Corollary~\ref{cor.growthRate}) we have the following three results, where $m=m(n)$ and we denote the average vertex degree in $\Gnm$ by $d=d(n) = 2m/n$.

\begin{prop} \label{prop.thm1.1}
(a) If $m \rightarrow \infty$ and $\, d \leq 1 +o(1)$ then $\; \q(\Gnm) \overset{p}\rightarrow 1$.

(b) Given constants $1< c_0 \leq c_1$,  there exists $\delta=\delta(c_0,c_1)>0$ such that if $c_0 \leq d \leq c_1$ for $n$ sufficiently large, then whp $\; \delta<\q(\Gnm)<1-\delta$. 

(c) If $d \rightarrow \infty$ then $\; \q(\Gnm) \overset{p}\rightarrow 0$.
\end{prop}
\smallskip

\begin{prop} \label{prop.thm1.2}
(i) If $1 \leq m = o(\sqrt{n})$ then whp $\, \q(\Gnm)= q_{\CC}(\Gnm) = 1- 1/m$.

(ii) 
If $m \geq 1$ satisfies $d = 2m/n \leq 1-  (\log n)^{1/2} n^{-1/4}$, then
\begin{equation*} \q(\Gnm) = q_{\CC}(\Gnm) = 1 -  \Theta(\frac1{m(1- d)}) \;\;\; \mbox{ whp}.\end{equation*}

(iii)
If $0 < \eps \leq 1/4$ and $1 \leq 2m/n \leq 1+\eps$ then $\q(G_{n,m}) \geq q_{\CC}(G_{n,m})> 1- (4\eps)^2$ whp.
\end{prop}

\begin{prop} \label{prop.thm1.3}
There exists $b$ such that for all $1 \leq m = m(n) \leq \binom{n}{2}$ we have
$\q(\Gnm)<\frac{b}{\sqrt{d}}$ whp.
 Also, given $0<\eps<1$, there exists $a=a(\eps) >0$ such that, if $m=m(n)$ satisfies $d \geq 1$  and $m \leq (1-\eps)\binom{n}{2}$ for $n$ sufficiently large,  then $\q(\Gnm) \geq \frac{a}{\sqrt{d}}$; and in this case we have
\begin{equation*} \frac{a}{\sqrt{d}} <\q(\Gnm)<\frac{b}{\sqrt{d}} \;\;\; \mbox{ whp}. \end{equation*}
\end{prop}

\needspace{6\baselineskip}
\subsection{Previous work on modularity}
\label{subsec.prev}
The vast majority of papers referencing modularity are papers in which real data, clustered using modularity-based algorithms, are analysed. Alongside its use in community detection, many interesting properties of modularity have been documented.
\medskip

\needspace{2\baselineskip}
\emph{Properties of modularity}

A basic observation is that, given a graph~$G$ without isolated vertices, in each optimal partition, for each part the corresponding induced subgraph of $G$ must be connected. The idea of a \emph{resolution limit} was introduced by Fortunato and Barth\'elemy~\cite{FortBart2008} in 2007: in particular, if a connected component~$C$ in an $m$-edge graph has strictly fewer than $\sqrt{2m}$ edges, then every optimal partition will cluster the vertices of~$C$ together.
This is so even if the connected component~$C$ consists of two large cliques joined by a single edge. This property highlights the sensitivity of optimal partitions to noise in the network: if that edge between the cliques, perhaps a mistake in the data, had not been there then the cliques would be in separate parts in every optimal partition. In contrast, although the structure of  optimal partitions is not robust to small changes in the edge set, the modularity \emph{value} of the graph is robust in this sense; see Section~\ref{sec.robustness}.

Concerning computational complexity, Brandes et al.\ showed that finding the (maximum) modularity of a graph is NP-hard~\cite{clusterduck}.  Furthermore it is NP-hard to approximate modularity to within any constant multiplicative factor~\cite{dinh2015network}. Modularity maximisation is also $W[1]$-hard, a measure of hardness in parameterised complexity, when parameterised by pathwidth; but approximating modularity to within multiplicative error $1\pm\eps$ is fixed parameter tractable when parameterised by treewidth~\cite{meeks2019parameterised}. The reduction in~\cite{clusterduck} required some properties of optimal partitions; for example it was shown that a vertex of degree~1 will be placed in the same part as its neighbour in every optimal partition.  Indeed, every part in every optimal partition has size at least 2 or is an isolated vertex, see Lemma~1.6.5 in~\cite{thesis}. The paper~\cite{clusterduck} also began the rigorous study of the modularity of classes of graphs, in particular of cycles and complete graphs. Later Bagrow~\cite{bagrow} and Montgolfier et al.\ \cite{modgraphclasses} proved that some classes of trees have high modularity, and this was extended in~\cite{treelike} to all trees with maximum degree $o(n)$, and indeed to all graphs where the product of treewidth and maximum degree grows more slowly than the number of edges.
There is a growing literature concerning the modularity behaviour of graphs in different classes, which we summarise in the appendix. 

Ehrhardt and Wolfe in~\cite{beatejournal} look at the distribution of the modularity score of a random partition of a graph or random graph. They consider different random weighted models, including the Erd\H{o}s-R\'enyi random graph $\Gnp$ for $np\rightarrow \infty$, where they show that the modularity score of a random partition is asymptotically normally distributed.  They do not investigate $\q(\Gnp)$.

\medskip

\needspace{2\baselineskip}
\emph{Statistical Physics predictions}

In 2004 Guimera et al.\ \cite{GPA04} observed through simulations that the modularity of random graphs can be surprisingly high. They conjectured that, for each (large) constant $c>1$,  
whp $\q({\Gncntext})\approx c^{-2/3}$. In 2006 Reichardt and Bornholdt~\cite{trulymodular} made a different conjecture for the modularity in this range. They assumed that an optimal partition will have parts of equal size, then approximated the number of edges between parts, using spin glass predictions from~\cite{kanter1987graph} for the minimum number of cross-edges in a balanced partition of a random graph, and predicted $\q(\Gncntext)\approx 0.97 \, c^{-1/2}$ whp. We confirm this growth rate. Indeed they predicted $\q(\Gnp)\approx 0.97\sqrt{(1-p)/{np}}$, which is $\Theta((np)^{-1/2})$, for $1/n \leq p\leq 0.99$. Thus Corollary~\ref{cor.growthRate} shows that, for a wide range of probabilities~$p$, the prediction of Reichardt and Bornholdt~\cite{trulymodular} is correct up to constant factors (and refutes that of Guimera et al.).
\medskip

Finally let us note that the conference version~\cite{modERAofA} contained some of the same material as the current full paper, but deferred several proofs and results to this paper.

\needspace{6\baselineskip}
\subsection{Plan of the paper}
The three phases theorem, Theorem~\ref{thm.usER}, gives an overview of the behaviour of the modularity $\q(\Gnp)$, with the three parts (a), (b) and (c) corresponding to increasing edge-probability $p$, starting with the sparse case.  The next two results, Theorems~\ref{thm.sparse} and \ref{thm.growthRate} (together with Corollary~\ref{cor.growthRate}) fill in many more details. Section~\ref{subsec.Gnm} contains corresponding results for the random graph $G_{n,m}$, similarly organised, starting with the sparse case. Our proofs are naturally organised in a similar way, starting with the sparse case.  

In Section~\ref{subsec.subcrit} we prove Theorem~\ref{thm.usER} part~(a) (by showing that $q_{\CC}(\Gnp) \overset{p}\rightarrow 1$ in the sparse case), and prove Theorem~\ref{thm.sparse}.   We prove Theorem~\ref{thm.usER} part~(b) in Section~\ref{sec.mid}: to prove the upper bound we use expansion properties of the giant component.
 Section~\ref{subsec.lb} concerns the $a(np)^{-1/2}$ lower bound on $\q(\Gnp)$; and indeed Theorem~\ref{thm.lowersqrt}gives a more detailed algorithmic version of the lower bound in Theorem~\ref{thm.growthRate}.  The proof involves analysing the algorithm \emph{Swap}, which starts with the odd-even bisection, and improves it by swapping certain pairs of vertices, increasing the edge contribution suitably without affecting the distribution of the degree tax.
 
Section~\ref{sec.robustness} contains robustness results for modularity, showing that when we change a few edges in a graph $G$ the modularity $\q(G)$ does not change too much. 
In Section~\ref{subsec.ub} we prove the upper bound $b(np)^{-1/2}$ on $\q(\Gnp)$ in Theorem~\ref{thm.growthRate}.  To do this, we first give a deterministic spectral upper bound on modularity, Lemma~\ref{lem.spectralmod}: we then complete the proof by using this bound and a robustness lemma from Section 5, together with results of Coja-Oglan~\cite{amin} and Chung, Vu and Lu~\cite{chunglv} on random graphs. 
In Section~\ref{sec.qconc} we give some results on the concentration and expectation of $\q(\Gnp)$.
Section~\ref{sec.concl} contains some concluding remarks and two conjectures.

There are also two appendices. Appendix A gives proofs for our results on the modularity~$\q(\Gnm)$ (stated in Section~\ref{subsec.Gnm}), and Appendix B gives a summary of some known modularities.   The inclusion of Appendix A is the main difference between this arXiv version and the journal version of the paper~\cite{ERusJournal}.  Indeed, apart from also giving some more details in the proof of Lemma~\ref{lem.qCC}, there are essentially no other differences between the two versions.


\needspace{12\baselineskip}
\section{The sparse phase: proofs of Theorem~\ref{thm.usER} (a) and Theorem~\ref{thm.sparse}}
\label{subsec.subcrit}

We can show that sufficiently sparse random graphs whp have modularity near 1 without developing any extra theory, and we do so here. This `near 1' modularity for sparse $p$ forms part (a) of the three phases theorem, Theorem~\ref{thm.usER}, and part (iii) of Theorem~\ref{thm.sparse}, and we prove these results in Section~\ref{subsec.sparseA}. More detailed results for modularity in the sparse range, forming parts (i) and (ii) of Theorem~\ref{thm.sparse}, are proved in Section 
\ref{subsec.sparseC}.

\needspace{6\baselineskip}
\subsection{Proofs of Theorem~\ref{thm.usER} (a) and Theorem~\ref{thm.sparse} (iii)}
\label{subsec.sparseA}

It is convenient to record first one standard preliminary result on degree tax.
\begin{lemma}\label{lem.degtax}
  Let the graph $G$ have $m \geq 1$ edges, and let $\cA$ be a $k$-part vertex partition for some $k \geq 2$. Then $q^D_\cA(G) \geq 1/k$; and if $x, y$ are respectively the largest, second largest volume of a part, then  $(x/2m)^2\leq q^D_\AA(G)\leq x/2m$ and  $q^D_\AA(G)\leq (x/2m)^2 + y/2m$.\end{lemma}
\begin{proof}
All the bounds follow from the convexity of $f(t)=t^2$.
Let $x_i$ be the volume of the $i$th part in $\cA$.
For the $1/k$ lower bound, observe that $x_1,\ldots,x_k \geq 0$ and $\sum_{i=1}^k x_i=2m$ together imply that $\sum_{i=1}^k x_i^2 \geq k\, (2m/k)^2= (2m)^2/k \,$; and
thus 
$q_{\cA}^D(G)  = \sum_i x_i^2 / (2m)^2 \geq 1/k$. 

For the upper bounds, observe that $0 \leq x_1,\ldots, x_k = x$ and $\sum_{i=1}^k x_i=2m$ together imply that $x^2 \leq \sum_{i=1}^k x_i^2 \leq (2m/x) \, x^2 = 2m x$; and so  $x^2/(2m)^2 \leq q^D_\cA(G)\leq x/2m$.
Similarly, supposing that $x_k=x$ and $x_{i}\leq y$ for $i=1,\ldots,k-1$,  we have
$\sum_{i=1}^{k-1} x_i^2 \leq (2m-x) y \leq 2my$; and so
$q^D_\cA(G)\leq (x^2+ 2my)/(2m)^2 =  (x/2m)^2 + y/2m$.
\end{proof}

Next we consider the connected components partition $\CC$, which has edge contribution 1.
\begin{lemma}\label{lem.qCC}
Let $0< \eps < 1$, let $c=1+\eps$, and let $p=(c +o(1))/n$. 
Then whp
\begin{equation*}1- \frac{16\eps^2}{(1+\eps)^4}  \; < \; q_{\CC}(\Gnp) \; < \; 1- \frac{16\eps^2}{(1+\eps)^4} \, (1-\sqrt{\eps}).\end{equation*} 
\end{lemma}

\begin{proof}
Let $f(x)=xe^{-x}$ for $x>0$, and note that $f$ is strictly increasing on $(0,1)$ and strictly decreasing on $(1,\infty)$.  Let $x=x(c)$ be the unique root in $(0,1)$ to $f(x)=f(c)$.
Let $m=e(\Gnp)$, and let $X$ be the maximum number of edges in a connected component of $\Gnp$, so $q_{\CC}^D(\Gnp) \geq (X/m)^2$.  Now whp $m \sim cn/2$ (that is, $m=(1+o(1))cn/2$), $X \sim ( 1 - x^2/c^2) c \, n/2$, and each component other than the giant has $O(\log n)$ edges, see for example Theorem~2.14 of \cite{frieze2015book}. Hence, by Lemma~\ref{lem.degtax},
\begin{equation*} q_{\CC}(\Gnp) = 1- (X/m)^2 + O((\log n)/n)  \;\; \mbox{ whp};\end{equation*}
and so
\begin{equation} \label{eqn.C}
 q_{\CC}(\Gnp) = 1- ( 1 - x^2/c^2)^2 + o(1)  \;\; \mbox{ whp}.
\end{equation}

\smallskip

\emph{Lower bound}

We claim that 
\begin{equation} \label{eqn.Xlb}
  1 - x^2/c^2 < 4\eps/(1+\eps)^2.
\end{equation}
To see this, let 
$g(t)= f(1+t) - f(1-t)$ for $0<t < 1$.  Now $g(0)=0$; and for $t >0$, \begin{equation*}g'(t) = t e^{-1}(e^{t} - e^{-t}) >0;\end{equation*}
and so $g(t)>0$ for all $0<t<1$.
Thus $g(\eps)>0$, that is 
$f(1\!-\!\eps) < f(c) = f(x)$,
and so $x>1-\eps$.
 Hence,
\begin{equation*}  1 - x^2/c^2 \, < \,  1- (1\!-\!\eps)^2/(1\!+\!\eps)^2
= 4 \eps/(1+\eps)^2, \end{equation*}
and we have proved~(\ref{eqn.Xlb}). 
Thus by~(\ref{eqn.C})
\begin{equation*} q_{\CC}(\Gnp) > 1- \tfrac{16 \eps^2}{(1+\eps)^4} \;\; \mbox{ whp}. \end{equation*}

\emph{Upper bound}

We claim that
\begin{equation}\label{eq.xupperbound}
   x < 1-\eps+\eps^{3/2}.
\end{equation}  
To check this, let 
$h(t)=(1+t)e^{-t}-(1- (t -t^{3/2}))e^{t-t^{3/2}}$
 for $0<t \leq 1$. 
 We want to show that $h(t)<0$ for $0<t \leq 1$: but $h(0)=0$, and 
{so it suffices to show that $h'(t)<0$ for $0<t<1$. But
\begin{equation*} h'(t)=-te^{-t}+te^{t -t^{3/2}}(1-\sqrt{t})(1-\tfrac{3}{2}\sqrt{t})
= te^{t -t^{3/2}} \left( - e^{-2t+t^{3/2}} + (1- \tfrac52 \sqrt{t} + \tfrac32 t) \right).\end{equation*}
Noting that $e^{-2t+t^{3/2}}> 1\!-\!2t+t^{3/2}$, we see that it suffices to show that, for $0<t<1$, 
\begin{equation*} 0< (1\!-\!2t+t^{3/2}) - (1- \tfrac52 \sqrt{t} + \tfrac32 t ) 
=  \sqrt{t}\, (\tfrac52 - \tfrac72 \sqrt{t} +t)  = \sqrt{t}( 1-\sqrt{t})(\tfrac52 - \sqrt{t});
\end{equation*}}
and 
the claim~\eqref{eq.xupperbound} follows. 
Hence,
\begin{equation*}
 c^2 - x^2 \; > \;  (1+\eps)^2 - (1- \eps + \eps^{3/2})^2 
 \; = \;   (2 + \eps^{3/2}) ( 2 \eps - \eps^{3/2})
  \; > \;  4 \eps - 2 \eps^{3/2}
\end{equation*}
so
\begin{equation*}1-\frac{x^2}{c^2} > \frac{4 \eps }{(1+ \eps)^2} \, (1- \tfrac12 \sqrt{\eps}).\end{equation*}
Finally, since $(1- \tfrac12 \sqrt{\eps})^2 > 1-\sqrt{\eps}$, 
we may use~(\ref{eqn.C}) to complete the proof.
\end{proof}

The final lemma in this subsection immediately implies both Theorem~\ref{thm.usER} (a) and Theorem~\ref{thm.sparse}~(iii).
\begin{lemma}\label{lem.subER}
Let $0<\eps \leq 1/4$, and let $p=p(n)$ satisfy $n^2p\rightarrow \infty$ and $np \leq 1+ \eps$ for $n$ sufficiently large. Then  $\q(\Gnp) \geq q_{\CC}(\Gnp) >1- (4\eps)^2$ whp.
\end{lemma}
\begin{proof}
Let $m=e(\Gnp)$, and let $X$ be the maximum number of edges in a connected component of $\Gnp$. For the connected components partition $\CC$, the edge contribution is 1, and so by Lemma~\ref{lem.degtax}, we have $q_\CC(\Gnp) \geq 1 - \frac{X}{m}$. We shall see that when $np \leq 1$ we have $X/m=o(1)$ whp, and so $q_{\CC}(\Gnp) = 1-o(1)$ whp. To prove this we break into separate ranges of $p$.  
Observe that since $n^2p\rightarrow\infty$ we have $m = (\tfrac12 +o(1)) n^2p$ whp. 

\noindent\textbf{Range 1: $n^2p\rightarrow \infty$ and $np\leq n^{-3/4}$.}
Whp $\Gnp$ consists of disjoint edges, since the expected number of paths on three vertices is $\Theta(n^3p^2)$. Hence whp $X/m=1/m=o(1)$.

\noindent\textbf{Range 2: $n^{-3/4} \leq np \leq 1/2$.}
Whp all components are trees or unicyclic and have $O(\log n)$ vertices. Hence whp $X=O(\log n)$ and whp
$X/m =O \left(\log n / n^2 p \right)=o(1)$.

\noindent\textbf{Range 3: $1/2 \leq   np\leq 1$.}
Since $np \leq 1$, whp $X = o(n)$ (see for example Theorem 5.19 of~\cite{jbook}).  
But whp $m = \Theta(n)$, and so whp $X/m=o(1)$.

\noindent\textbf{Range 4: $1 <  np \leq 1+ \eps$} (where $0< \eps \leq 1/4$ is fixed).
Let $c = 1+\eps$. For $G_{n,c/n}$, whp $X = (1+o(1)) \, ( 1 - x^2/c^2) c \, n/2$; and, uniformly over $1/n \leq p \leq c/n$, each component other than the largest has $o(n)$ edges (see for example Theorem 2.14 of~\cite{frieze2015book}). 
Hence, by inequality~\eqref{eqn.Xlb}, for $G_{n,c/n}$, whp $X \leq \tfrac{4 \eps}{1+ \eps} \, \tfrac{n}{2}$; and so by monotonocity this holds also for $\Gnp$ (with $p \leq cn$ as here).  Also,  $e(G_{n,1/n}) \geq \tfrac{1+\eps/2}{1+\eps} \tfrac{n}{2}$ whp, and so by monotonocity this holds also for $\Gnp$.  Now by Lemma~\ref{lem.degtax}, whp 
\begin{equation*} q_\CC(\Gnp) \geq 1 - (X/m)^2 - o(1) \geq 1- (4\eps)^2/(1\!+\!\eps/2)^2 - o(1) > 1- (4\eps)^2.\end{equation*}
This completes the proof of the lemma.
\end{proof}

\needspace{6\baselineskip}
\subsection{Proof of Theorem~\ref{thm.sparse} (i), (ii)}\label{subsec.sparseC}

Part~(iii) of Theorem~\ref{thm.sparse} was proved in the last subsection: the next two results will allow us to complete the proof of Theorem~\ref{thm.sparse}. The resolution limit theorem of Fortunato and Barth\'elemy~\cite{FortBart2008} described in Section \ref{subsec.prev} immediately gives the following lemma.

\begin{lemma} Let $G$ consist of $m\geq 1$ isolated edges and perhaps some isolated vertices. Then, ignoring isolated vertices, the connected components partition $\CC$ is the unique optimal partition and $\q(G)=q_\CC(G)=1-1/m$.
\end{lemma}
\smallskip

\begin{prop}\label{prop.CCbestonly}
Suppose that $n^2p\rightarrow \infty$ and $np \leq 1- \gamma$, where $\gamma=\gamma(n)=(\log n)^{1/2} n^{-1/4}$.
Then for $\Gnp$ whp the connected components partition is the unique optimal partition (up to shuffling of isolated vertices).   
\end{prop}
\begin{proof}
Let $m$ be the (random) number of edges in $\Gnp$. Let $X$ be the maximum number of edges in a component.
By the resolution limit result mentioned above, it suffices to show that whp $X < \sqrt{2m}$. Let $L_1$ be the maximum number of vertices in a component.  Then $X \leq L_1$ whp, by for example Theorem 5.5 of \cite{jbook}.
Consider the following three overlapping ranges for~$p$: firstly $n^2p\rightarrow \infty$ and $n^{3/2}p \rightarrow 0$, secondly $n^{7/4}p \rightarrow \infty$ and $np \leq 1/2$, and finally $1/2 \leq np \leq 1-\gamma$.

For $p$ in the first range, whp $m\sim n^2p/2 \rightarrow \infty$, and by a first moment argument, whp $\Gnp$ 
consists of isolated vertices and disjoint edges; and so whp $X \leq 1<\sqrt{2m}$. 
Secondly, for $p$ such that $n^{7/4}p \rightarrow \infty$ and $np \leq c$, whp $m\geq n^{1/4}$, and whp $L_1$ is $O(\log n)$ by for example Theorem 5.4 of~\cite{jbook}; and so again $X <\sqrt{2m}$ whp.

Finally suppose that $1/2 \leq np \leq 1-\gamma$.  Then $m \leq \frac{n}2 - (\frac12 + o(1))(\log n)^{1/2} n^{3/4}$ whp, so by Theorem~5.6 of~\cite{jbook} we have $L_1 \leq (\frac18 +o(1)) \sqrt{n}$ whp.  But $2m \geq 0.49 n$ whp, so $X \leq L_1 < \sqrt{2m}$ whp.
This completes the proof.
\end{proof}

\begin{proof}[Proof of Theorem~\ref{thm.sparse} \, part (i)]
As in the proof of Lemma~\ref{lem.subER}, whp $\Gnp$ consists of isolated vertices and disjoint edges.  But if such a graph $H$ has $m \geq 1$ edges, then (ignoring isolated vertices) the unique optimal partition is the connected components partition $\mathcal C$, and $q_{\mathcal C}(H) = 1-1/m$. 
\end{proof}

\begin{proof}[Proof of Theorem~\ref{thm.sparse} \, part (ii) ]
Let $G=\Gnp$ have connected components $C_1, C_2, \ldots$. When we write $\sum_i$ here, we mean the sum over all components $C_i$. By Proposition~\ref{prop.CCbestonly}, whp the connected components partition $\mathcal C$ is optimal, that is $\q(G)=q_\mathcal{C}(G)$.  Since $q_\mathcal{C}(G)=1-q^D_\mathcal{C}(G)$, it suffices to show that whp $\frac{1}{m^2}\sum_i e(C_i)^2 = \Theta(\frac{1}{n^2p(1-np)})$; and since whp $m=\Theta(n^2p)$, this is equivalent to showing that whp $\sum_i e(C_i)^2 = \Theta(\frac{n^2p}{1-np})$.

We split into two overlapping ranges of $p$.  Call a connected component with at least one edge a \emph{non-trivial} component.  Suppose first that $n^2p\rightarrow \infty$ and $n^{10/9}p\rightarrow 0$.  Whp each non-trivial connected component is a tree with between two and nine vertices, by the first moment method. (There are also whp isolated vertices.) In particular, whp for each non-trivial connected component~$C_i$ we have $1 \leq e(C_i)\leq 8$. Thus whp $m \leq \sum_{i\in I} e(C_i)^2 \leq 8m $ and so whp $\sum_i e(C_i)^2=\Theta(n^2p) = \Theta(\frac{n^2p}{1-np})$, as required. From now on suppose that $n^{9/8}p\rightarrow \infty$ and $np\leq 1-(\log n)^{1/2}n^{-1/4}$.

Let $I$ index the tree components, and let $J$ index the unicyclic components. By Theorem~5.5 of~\cite{jbook} whp there are no complex components, so whp $I$ and $J$ index all components and $|I|=m-n$.

Hence whp
\begin{eqnarray} 
\label{eq.IJsum}\sum_{i} e(C_i)^2  = 
\sum_{i\in I} (|C_i|-1)^2 + \sum_{i \in J} |C_i|^2 
 = 
\sum_{i \in I \cup J} |C_i|^2 - 2\sum_{i \in I}|C_i| + |I| .
\end{eqnarray}

But now since $\sum_{i\in I}|C_i| \geq |I|$ and whp $|I|= n-m$ we have whp $2\sum_i |C_i|-|I|\geq n-m$. On the other hand, since $\sum_i |C_i|\leq n$,  whp $2\sum_i |C_i|-|I| \leq 2n-|I|\leq n+m$. This together with \eqref{eq.IJsum} implies that whp
\begin{equation}\label{eq.IJsum2}
\sum_{i \in I \cup J} |C_i|^2-n-m  \leq  \sum_{i} e(C_i)^2 \leq 
\sum_{i \in I \cup J} |C_i|^2- n+m .
\end{equation}

By Theorem~1.1 of Janson and Luczak~\cite{JansonSuss} concerning susceptibility, whp 
\begin{equation*}\sum_i |C_i|^2 = \frac{1}{1-np}\Big(n+O\Big(\frac{n^{1/2}}{(1-np)^{3/2}}\Big)\Big).\end{equation*}
Hence whp \begin{equation}\label{eq.susseq} \sum_i |C_i|^2 - n 
= \frac{1}{1-np}(n^2p+O(n^{7/8}))= \frac{n^2p}{1-np}(1+o(1)),
\end{equation}
where we used $1-np \geq n^{-1/4}$ in the first step, and $n^{9/8}p\rightarrow \infty$ to imply that $n^{7/8}=o(n^2p)$ in the second step.

We are almost done. Recall that it suffices to show whp $\sum_i e(C_i)^2 =\Theta(\frac{n^2p}{1-np})$ to finish the proof. But by~\eqref{eq.IJsum2} the expression in \eqref{eq.susseq} differs from $\sum_i e(C_i)^2$ by at most $m$. Now whp $m=\frac{1}{2}n^2p(1+o(1))$, so whp changing the value of \eqref{eq.susseq} by at most $m$ will not change the order of the leading term, and thus we have $\sum_i e(C_i)^2=\Theta(\frac{n^2p}{1-np})$ whp.
\end{proof}


\needspace{12\baselineskip}
\section{The middle phase: proof of Theorem~\ref{thm.usER} (b)}
\label{sec.mid}

It is straightforward to use known results to prove Theorem~\ref{thm.usER} part (b).
First we show that the connected components partition $\CC$ yields the lower bound.  As we noted earlier, the lower bound will follow also from the lower bound in Theorem~\ref{thm.growthRate}, but that has quite an involved proof, whereas the proof below is only a few lines. Also as we noted earlier, the upper bound in Theorem~\ref{thm.growthRate} will give the upper bound in Theorem~\ref{thm.usER} part (b) for large $np$, but not when $np$ is small.  

\needspace{6\baselineskip}
\subsection{Proof of lower bound}
There is a simple reason why the modularity $\q(\Gnp)$ is bounded away from 0 whp when the average degree is bounded, namely that whp there is a linear number of isolated edges.  First, here is a deterministic lemma.

\begin{lemma} \label{lem.isoledges}
Let the graph $G$ have $m \geq 2$ edges, and $i \geq \eta m$ isolated edges, where $0<\eta \leq \frac12$.  Then $q_{\CC}(G) \geq \eta$.
\end{lemma}
\begin{proof}
Note first that if $i=m$ then $q_\CC(G)=1-1/m \geq \eta$. Thus we may assume that $i<m$, and so $i \leq m-2$.
Since there are in total $m-i$ edges in the components which are not isolated edges,
\begin{equation*} q_\CC(G) \geq 1 - \frac{(m-i)^2}{m^2} - \frac{i}{m^2}.\end{equation*}
Treating $i$ as a continuous variable and differentiating, we see that the bound is an increasing function of $i$ for $i  \leq m-1$; and so, setting $i=\eta m$,
\begin{equation*} q_\CC(G) \geq 1-(1-\eta)^2 - \eta/m = \eta + \eta(1 -\eta - 1/m) \geq \eta, \end{equation*}
as required.
\end{proof}

Assume that $1 \leq np \leq c_1$.
Let $X$ be the number of isolated edges in $\Gnp$.  Then \begin{equation*} \E[X] = \binom{n}{2}p (1-p)^{2n-4} = n \cdot (\tfrac12 +o(1)) np \, e^{-2np} \geq n \cdot (\tfrac12 +o(1)) c_1 e^{-2c_1},\end{equation*}
since $f(x)= xe^{-2x}$ is decreasing for $x > \frac12$.
A simple calculation shows that the variance of $X$ is $o((\E[X])^2)$: thus by Chebyshev's inequality, whp 
$X \geq n \cdot \tfrac13 c_1 e^{-2c_1}$.
Similarly, whp $m= e(\Gnp) \leq \tfrac23 c_1 n$; and so whp $X/m \geq \tfrac12 e^{-2c_1}$.
Finally, Lemma~\ref{lem.isoledges} shows that whp $q_\CC(\Gnp) \geq \eta= \tfrac12 e^{-2c_1}$.  This completes the proof of the lower bound in Theorem~\ref{thm.usER}(b).

\needspace{6\baselineskip}
\subsection{Proof of upper bound}
It is convenient to spell out the upper bound in Theorem~\ref{thm.usER}(b) as the following lemma.
\begin{lemma}\label{lem.upperEps}
Given constants $1< c_0 \leq c_1$, there exists $\eps=\eps(c_0,c_1)>0$ such that, if $c_0 \leq np \leq c_1$ for $n$ sufficiently large, then whp $\q(\Gnp) < 1-\eps$.
\end{lemma}
The idea of the proof of this lemma is that if some part in a partition has large volume then the degree tax of the partition is large, and if all parts have small volume then the edge contribution must be small.
We use a result from \cite{mcdbisect} concerning edge expansion in the giant component. Define a $(\delta, \eta)$-cut of $G=(V,E)$ to be a bipartition of $V$ into $V_1, V_2$ such that both sets have at least $\delta|V|$ vertices and $e(V_1,V_2)<\eta|V|$. We need only the case $\delta=1/3$.

\begin{proof}[Proof of Lemma~\ref{lem.upperEps}]\label{proof.upperEps} 

We employ double exposure. Let $G'\sim \mathcal{G}_{n,{c_0/n}}$.	  For each non-edge of $G'$ resample with probability $p'=(p-c_0/n)/(1-c_0/n)$, to obtain $G$ such that $G \sim \Gnp$. Let $\cA$ be an optimal partition for $G$.  Observe that whp $m=e(G)<c_1n$, and then
\begin{equation*}1-\q(G)=\frac{1}{2m}\sum_{A\in \cA} \left( e_G(A,\bar{A})+\frac{\vol_G(A)^2}{2m} \right) 
> \frac{1}{2c_1n}\sum_{A\in \cA} \left( e_{G'}(A,\bar{A})+\frac{\vol_{G'}(A)^2}{2c_1n}\right).\end{equation*}
Thus it suffices to show  
that whp, for each vertex partition $\cA$,
\begin{equation}\label{eq.sufeps} 
\sum_{A\in \cA} \left(e_{G'}(A,\bar{A})+\frac{\vol_{G'}(A)^2}{2c_1n}\right) \geq 2\eps c_1 n.
\end{equation}
We will now work solely with $G'$, so we shall drop the subscripts. Whp $G'$ has a unique giant component~$H$, $H$ does not admit a $(1/3,\eta)$-cut for a constant $\eta=\eta(c_0)>0$ by \cite{mcdbisect} [Lemma~2], and $|V(H)| \sim (1-t_0/c_0)n$ where $t_0<1$ satisfies $t_0e^{-t_0}=c_0e^{-c_0}$ by~\cite{ERgiant}. Let $F$ be the event that $G'$ has a unique giant component $H$, $H$ does not admit a $(1/3,\eta)$-cut, and $|V(H)|\geq \frac12(1-t_0/c_0)n + 3$.
Then the event $F$ holds whp.  Let $W$ be a set of vertices such that $|W| \geq \frac12(1-t_0/c_0)n + 3$, and let $F_W$ be the event that $F$ holds and $V(H)=W$.
To prove the lemma, it suffices to show that, conditioning on $F_W$ holding, the inequality~\eqref{eq.sufeps} holds with
\begin{equation*}\eps =\min \{(1-t_0/c_0)^2/36c_1^2, \eta(1-t_0/c_0)/2c_1\}.\end{equation*}

Fix any graph $G'$ such that $F_W$ holds.
Let $\cA$ be any vertex partition which minimises the left side of~(\ref{eq.sufeps}).  It is easy to see that, for each part $A$ of $\cA$, the subgraph of $G'$ induced on $A$ must be connected.  Let $\mathcal{H}$ be the partition of the giant component $H$ induced by $\cA$; and note that $\mathcal{H}$ consists of the parts $A \in \cA$ with $A \cap W$ non-empty.
Relabel $\mathcal{H}$ as $\{W_1,\ldots,W_h\}$ where $h \geq 1$ and $|W_1|\geq \ldots \geq |W_h|$. 
There are two cases to consider.

\noindent \textbf{Case 1.} Suppose $|W_1|\geq |W|/3$. As the subgraph of $G'$ induced by $W_1$ is connected,
\begin{equation*} \vol(W_1) \geq 2(|W_1|-1) \geq (1-t_0/c_0)n /3;\end{equation*} 
and so
\begin{equation*}
\sum_{A\in \cA} \frac{\vol(A)^2}{2c_1n} \geq \frac{\vol(W_1)^2}{2c_1n} 
\geq \frac{(1-t_0/c_0)^2n^2}{18c_1 n} \geq 2 \eps c_1 n,\end{equation*}
which yields~\eqref{eq.sufeps}.

\noindent \textbf{Case 2.} Now suppose that $|W_i| < |W|/3$ for all parts $W_i$ of $\mathcal{H}$. We group the parts to make a bipartition $W=B_1\cup B_2$ with $B_1$ and $B_2$ of similar size.  We may for example start with $B_1$ and $B_2$ empty, consider the $W_i$ in turn, and each time add $W_i$ to the smaller of $B_1$ and $B_2$ (breaking ties arbitrarily). This clearly gives $||B_1|-|B_2|| < |W|/3$. 
Since there is no $(1/3, \eta)$-cut of $H$ in $G'$, we have $e(B_1, B_2) \geq \eta |W|$. But each edge between $B_1$ and $B_2$ lies between the parts of $\cA$, and so \begin{equation*}\sum_{A\in \cA} e(A,\bar{A})  \geq 2 e(B_1,B_2) \geq 2 \eta |W|
> \eta (1-t_0/c_0)n \geq 2 \eps c_1 n,\end{equation*} 
which again yields~\eqref{eq.sufeps}, and completes the proof.
\end{proof}

From the proof of Lemma 2 in~\cite{mcdbisect}, we may see that if $c=1+t$ for constant $t>0$, we can take $\eta = \Theta(t^2/\log(1/t))$ as $t \to 0$.
Now, arguing as for Lemma~\ref{lem.upperEps}, we obtain $\q(G_{n,(1+t)/n}) < 1 - \Theta(t^3/\log (1/t))$ whp.  Observe that this upper bound fails to match the lower bound $1- \Theta(t^2)$  in part(iii) of Theorem~\ref{thm.sparse}: this is discussed further in Section~\ref{sec.concl}.


\needspace{12\baselineskip}
\section{The $a(np)^{-1/2}$ lower bound on the modularity $\q(\Gnp)$}
\label{subsec.lb}
In the first subsection, we analyse a simple algorithm \emph{Swap} which, given a graph $G$, runs in linear time (in time $O(n+m)$ if $G$ has $n$ vertices and $m$ edges), and constructs a balanced bipartition $\cA$ of the vertices, such
that $q_{\cA}(\Gnp)$ yields a good lower bound on $\q(\Gnp)$ -- see Theorem~\ref{thm.lowersqrt}.
In the second subsection, we consider a smaller range of probabilities, and see that recent results on stochastic block models yield similar lower bounds, with better constants  -- see Theorem~\ref{thm.lowerSqrt2}.

\needspace{6\baselineskip}
\subsection{The algorithm \emph{Swap}}
\label{subsec.swap}

Given a graph $G$, the algorithm \emph{Swap} described below constructs a balanced bipartition $\cA$ of the vertices.  It runs in linear time (in time $O(n+m)$ if $G$ has $n$ vertices and $m$ edges).  

\begin{thm}\label{thm.lowersqrt}
There are constants $c_0$ and $a>0$ such that (a) if $p=p(n)$ satisfies $c_0 \leq np \leq n -c_0$ for $n$ sufficiently large, then whp $q_{\cA}(\Gnp) \geq \tfrac15 \sqrt{\frac{1-p}{np}}$; and (b) if $p=p(n)$ satisfies $1 \leq np \leq n -c_0$ for $n$ sufficiently large, then whp $q_{\cA}(\Gnp) \geq a \sqrt{\frac{1-p}{np}}$.
\end{thm}

The algorithm Swap starts with a balanced bipartition of the vertex set into $A \cup B$, which has modularity very near 0 whp.  By swapping some pairs $(a_i,b_i)$ between $A$ and $B$, whp we can increase the edge contribution significantly, without changing the distribution of the degree tax (and without introducing dependencies which would be hard to analyse).  
Before we start the main part of the proof of the theorem, it is convenient to give three elementary preliminary lemmas.  The first sets the scene, by considering a natural fixed bipartition.

\begin{lemma} \label{lem.orig}
Let $p=p(n)$ satisfy $1/n^2 \leq p \leq 1-1/n^2$, and consider $G = \Gnp$. Let $\cA$ be the bipartition of $V=[n]$ into $A=\{j \in V:j \mbox{ is odd}\}$ and $B=\{j \in V: j \mbox{ is even}\}$.  Let $\omega=\omega(n) \to \infty$ arbitrarily slowly as $n \to \infty$. Then whp $q_{\cA}^E(G) = \tfrac12 -\tfrac1{2n} + o\big(\tfrac{\omega}{n} \sqrt{\tfrac{1-p}{p}}\big)$ and $q_{\cA}^D(G) = \frac12 + O(\tfrac1{n^2}) + o\big(\tfrac{\omega (1-p)}{n^2p})$, so $q_{\cA}(G)= -\tfrac1{2n} + o\big(\tfrac{\omega}{n} \sqrt{\tfrac{1-p}{p}}\big)$.
\end{lemma}

\begin{proof}
Observe first that $e(G) \sim \Bin(\binom{n}{2},p)$, with mean $\binom{n}{2}p$ and variance less than $\frac12 n^2 p(1-p)$;  so 
\begin{equation} \label{eqn.etot}
   e(G) = \tfrac12 n^2 p -\tfrac12 np +o(\omega \sqrt{n^2 p(1-p)}) = \tfrac12 n^2 p \ \big(1-\tfrac1{n} + o\big(\tfrac{\omega}{n} \sqrt{\tfrac{1-p}{p}}\big) \big) \;\; \mbox{ whp}.
\end{equation}
Also, $e(A,B)$ has mean $n^2p/4$ if $n$ is even and $(n^2 -1)p/4$ if $n$ is odd; and has variance at most $n^2p(1-p)/4$.  Hence
\begin{equation} \label{eqn.ecross}
   e(A,B) = \tfrac14 n^2 p \big( 1 + o\big(\tfrac{\omega}{n} \sqrt{\tfrac{1-p}{p}}\big) \big) \;\; \mbox{ whp}
\end{equation}
and so by~(\ref{eqn.etot})
\begin{equation} \label{eqn.edgecont}
   q_{\cA}^E(G) = \tfrac12 -\tfrac1{2n} + o\big(\tfrac{\omega}{n} \sqrt{\tfrac{1-p}{p}}\big) \;\; \mbox{ whp}.
\end{equation}

Observe that $\vol(A)-\vol(B)$ has mean 0 if $n$ is even and mean $(n-1)p$ if $n$ is odd,
and has variance at most $n^2 p(1-p)$.  Let $\eta= |\vol(A)-\vol(B)|$. 
Then whp $\eta \leq np+ \omega^{1/4} n \sqrt{p(1-p)}$, and so 
 $\eta^2 \leq 2n^2p^2 + 2 \sqrt{\omega} n^2 p(1-p)$.
But $q_{\cA}^D(G) = \frac12 + \frac{\eta^2}{2\, \vol(G)^2}$, 
and so by~(\ref{eqn.etot}) 
\begin{equation} \label{eqn.degtax}
\tfrac12 \leq q_{\cA}^D \leq \tfrac12 + \tfrac2{n^2} + o\big(\tfrac{\omega (1-p)}{n^2p}) \;\; \mbox{ whp}. 
\end{equation}  
The final result for $q_{\cA}(G)$ follows directly from \eqref{eqn.edgecont} and~\eqref{eqn.degtax}.
\end{proof}

The second preliminary lemma concerns swapping values in certain symmetrical distributions.  It may seem intuitively clear, but there is a shortish proof so we give it below.

\begin{lemma} \label{lem.swap2}
Let the discrete random variables $X$ and $Y$ satisfy
$\pr((X,Y)=(a,b)) = \pr((X,Y)=(b,a))$ for all $a,b$; and let the $\{0,1\}$-valued random variable $J$ satisfy
\begin{equation*} \pr(J=1 \mid (X,Y)=(a,b)) = \pr(J=1 \mid (X,Y)=(b,a)) \end{equation*}
for all $a,b$ such that $\pr((X,Y)=(a,b))>0$. Define the random variables $X'$ and $Y'$ by setting $(X',Y')=(X,Y)$ if $J=0$, and 
$(X',Y')=(Y, X)$ if $J=1$ (that is, we swap when $J=1$). Then
  $ (X',Y') \sim (X,Y)$.
\end{lemma}
\begin{proof}
Fix $a,b$ such that $\pr((X,Y)=(a,b))>0$.  For $i=0,1$
\begin{eqnarray*}
\pr((X,Y)\!=\!(a,b), J\!=\!i) &=&
\pr((X,Y)\!=\!(a,b)) \, \pr( J\!=\!i \, | (X,Y)\!=\!(a,b))\\
&=&
\pr((X,Y)\!=\!(b,a)) \, \pr( J\!=\!i \, | (X,Y)\!=\!(b,a)) \; = \; \pr((X,Y)\!=\!(b,a), J\!=\!i)
\end{eqnarray*}
and so
\begin{eqnarray*}
\pr( (X',Y')=(a,b)) &=&
\pr( (X,Y)=(a,b), J=0) + \pr( (X,Y)=(b,a), J=1)\\
&=&
\pr( (X,Y)=(b,a), J=0) + \pr( (X,Y)=(a,b), J=1)\\
&=&
\pr( (X',Y')=(b,a)).
\end{eqnarray*}
Hence, 
\begin{eqnarray*}
2 \, \pr( (X',Y')=(a,b)) 
&=&
\pr( (X',Y')=(a,b)) + \pr( (X',Y')=(b,a))\\
&=&
\pr( (X,Y)=(a,b), J=0) + \pr( (X,Y)=(b,a), J=1)\\
&& + \pr( (X,Y)=(a,b), J=1) + \pr( (X,Y)=(b,a), J=0) \\
 &=&
\pr( (X,Y)=(a,b)) + \pr( (X,Y)=(b,a))\\
&=&
2 \, \pr( (X,Y)=(a,b)).
\end{eqnarray*}
It follows that $(X',Y') \sim (X,Y)$, as required.  
\end{proof}

The final preliminary lemma concerns the expected absolute value of the difference between two independent random variables with the same binomial distribution.

\begin{lemma} \label{lem.diff}
Given $0<\eps<1$, there is a $c_0$ such that the following holds. Let $p=p(n)$ satisfy $c_0 \leq np \leq n-c_0$ for $n$ sufficiently large.  For each $n$, let the random variables $X_n$ and $Y_n$ be independent, each with distribution $\Bin(n,p)$, and let $U_n=X_n-Y_n$.  Then
\begin{equation*} \E[|U_n|] \geq (1-\eps) \sqrt{4np(1-p)/\pi}\end{equation*}
for $n$ sufficiently large.  
\end{lemma}
\begin{proof}
Let $\sigma(n) = \big( 2 np(1-p) \big)^{1/2}$.
If $c_0 \leq np \leq n-c_0$, once $n \geq 2 c_0$ we have
\begin{equation} \label{eqn.sigma}
 \sigma(n) \geq 
\sqrt{2c_0(1-c_0/n)} \geq \sqrt{c_0}.
\end{equation}
We may write $U_n$ as $\sum_{i=1}^{n} Z_i$, where the $Z_i$ are iid $\{0, \pm 1\}$-valued, with $\pr(Z_1=1)=\pr(Z_1=-1) = p(1-p)$ (and $\pr(Z_1=0)=1-2p(1-p)$).  Note that $\E[Z_1]=0$, $\tau^2 := \var(Z_1) = 2p(1-p)$, and $\E[|Z_1|^3]= \tau^2$; and note also that $\sigma(n) =\sqrt{n} \tau$.
Let $\tilde{U}_n= U_n/\sigma(n)$.
By the Berry-Esseen theorem, for all real~$x$
\begin{equation} \label{eqn.be}
 \big| \pr(\tilde{U}_n \leq x) - \Phi(x) \big| \leq \frac{ C \, \E[|Z_1|^3]}{\tau^3 \sqrt{n}} = \frac{C}{\sigma(n)}
\end{equation}
where we may take the constant $C$ as $1/2$.

Let $X \sim N(0,1)$. 
Let $\eta>0$.  There is a $b$ such that
$\E[X \ind_{0 \leq X \leq b}] \geq \E[X \ind_{X \geq 0}] -\eta$.
By~(\ref{eqn.sigma}) and~(\ref{eqn.be}), there is a $c_0$ sufficiently large that, when $c_0 \leq np \leq n-c_0$,
\begin{equation*} \big| \pr(\tilde{U}_n \geq x) - (1-\Phi(x)) \big|  \leq \eta/b,\end{equation*}
and so in particular, for all $x > 0$
\begin{equation*} \pr(\tilde{U}_n  \geq x) \geq \pr(X \geq x) - \eta/b.\end{equation*}
Hence
\begin{equation*} \E[\tilde{U}_n \ind_{\tilde{U}_n \geq 0}] \geq
\int_{0}^{b} (\pr(X \geq x) - \eta/b)\, dx
= \E[X \ind_{0 \leq X \leq b}] - \eta \geq \E[X \ind_{X \geq 0}] - 2\eta.
\end{equation*}
Thus we have
\begin{equation*} \E[|\tilde{U}_n|] = 2 \E[\tilde{U}_n \ind_{\tilde{U}_n \geq 0}] \geq 2 \E[X \ind_{X \geq 0}] - 4\eta = \E[|X|] - 4 \eta = \sqrt{2/\pi} -4\eta \, ; \end{equation*}
and so 
\begin{equation*} \E[U_n] \geq (\sqrt{2/\pi}-4\eta)\, \sigma(n), \end{equation*}
which yields the lemma. 
\end{proof}

\begin{proof}[Proof of Theorem~\ref{thm.lowersqrt}] Let $n \geq 6$, and let $V=[n]$.  We start with the initial bipartition $\cA$ of $V$ into $A=\{j \in V:j \mbox{ is odd}\}$ and $B=\{j \in V: j \mbox{ is even}\}$, as in Lemma~\ref{lem.orig}.  Let $k=k(n)=\lfloor n/6 \rfloor$. 
Let $V_0=[4k]$, let $V_1= \{4k+1,\ldots,6k\}$ and let $V_2=\{6k+1,\ldots,n\}$.  Note that $0 \leq |V_2| \leq 5$: we shall essentially ignore any vertices in $V_2$.
Let $A_i = A \cap V_i$ and $B_i=B \cap V_i$ for $i=0,1,2$.  The six sets $A_i, B_i$ are pairwise disjoint with union $V$. 
Currently $V_0$ is partitioned into $A_0 \cup B_0$: the algorithm Swap `improves' this partition, 
 keeping the other 4 sets fixed. For $i=1,\ldots,2k\,$ let $a_i=2i-1$ and $b_i=2i$,
so $A_0=\{a_1,\ldots,a_{2k}\}$ and $B_0=\{b_1,\ldots,b_{2k}\}$.
The way that we improve the partition $V_0 = A_0 \cup B_0$ is by swapping $a_i$ and $b_i$ for certain values $i$.

Consider the initial bipartition $\cA$. 
Write $G$ for $\Gnp$.  By Lemma~\ref{lem.orig}, whp $q_{\cA}(G)$ is very near 0.
For each $i \in [2k]$ let
\begin{equation*}T_i =  e(a_i, B_1) - e(a_i, A_1) + e(b_i, A_1) - e(b_i, B_1),\end{equation*}
and note that the random variables $T_1,\ldots,T_{2k}$ are iid.
Observe that if $T_i>0$ and we swap $a_i$ and $b_i$ between $A_0$ and $B_0$ (that is, replace $A_0$ by $(A_0 \setminus\{a_i\}) \cup \{b_i\}$ and similarly for $B_0$) then $e(A,B)$ decreases by $T_i$, so the edge contribution of the partition increases -- see Figure~\ref{fig.swapping}.  The algorithm Swap makes all such swaps (looking only at possible edges between $V_0$ and $V_1$). 
For each $i \in [2k]$, let $(a'_i,b'_i)=(b_i,a_i)$ if we perform a swap, and let $(a'_i,b'_i)=(a_i,b_i)$ if not; and let $A'_0=\{a'_1,\ldots,a'_{2k}\}$ and $B'_0=\{b'_1,\ldots,b'_{2k}\}$.
 Let us call the resulting balanced bipartition $\cA' =(A',B')$,  
 where $A'=A'_0 \cup A_1 \cup A_2$ and $B'=B'_0 \cup B_1 \cup B_2$. 
We shall see that $q_{\cA'}(G)$ is as required.   

Let $T^*= \sum_{i \in [2k]}|T_i|$.  Observe that
\begin{equation*} e(A'_0, A_1)+e(B'_0, B_1) - (e(A'_0, B_1) + e(A_1, B'_0)) = T^*.\end{equation*}
But
\begin{equation*} e(A'_0, A_1)+e(B'_0, B_1) + (e(A'_0, B_1) + e(A_1, B'_0)) = e(V_0,V_1),\end{equation*}
so
\begin{equation} \label{eqn.gain}
  e(A'_0,B_1)+e(A_1,B'_0)  = \tfrac12 e(V_0,V_1) - \tfrac12 T^*. 
\end{equation}
This is where $\cA'$ will gain over $\cA$. We shall show that whp $T^*$ is large, see inequality~(\ref{eqn.tstarbignew}).  However, before that, we show quickly that the degree tax for $\cA'$ has exactly the same distribution as for the initial bipartition $\cA$, and so it is very close to 1/2 whp.  Let $\omega = \omega(n) \to \infty$ (arbitrarily slowly) as $n \to \infty$.
\medskip

\begin{figure}[h]\hspace{4mm}
	\begin{tikzpicture}[scale=0.95]
				\draw[color=red, thin] (1.6,2.65) ellipse (0.9 and 1.35);
				\node[red] at (2.8,2.35) {$A_1$};
				
				\draw[color=blue, thin] (5.1,2.65) ellipse (0.9 and 1.35);
				\node[blue] at (6.3,2.35) {$B_1$};
		
				\filldraw[color=red, fill=red](1.5,6) circle [radius=0.1];
				\node [right] at (1.6,6) {\small $a_i$};
				\filldraw[color=blue, fill=blue](5,6) circle [radius=0.1];
				\node [right] at (5.1,6) {\small $b_i$};

				\draw[color=black, opacity=0.2] (1.6,6) ellipse (0.96 and 1.9);
				\node[color=black, opacity=0.2] at (2.9,5.75) {$A_0$};
				\draw[color=black, opacity=0.2] (5.1,6) ellipse (0.96 and 1.9);
				\node[color=black, opacity=0.2] at (6.4,5.75) {$B_0$};

				\path [fill=red] (1.5,5.8) -- (1.65,4.1) -- (1.35,4.1);    
				\path [fill=blue] (5,5.8) -- (5.5,4.1) -- (4.7,4.1);
				\path [fill=gray, opacity=0.8] (1.6,5.8) -- (3.9,2.6) -- (4.4,3.9);   
				\path [fill=gray, opacity=0.8] (4.9,5.8) -- (2.86,3.1) -- (2.6,3.8);

				\draw[->,very thick] (7.8,4.7) -- (9.1,4.7);
				\node at (8.5,5.55)  {swap if $T_i >0$}; 

				\begin{scope}[shift={(10,0)}]
				\draw[color=red] (1.6,2.65) ellipse (0.9 and 1.35);
				\node[red] at (2.8,2.35) {$A_1$};
				
				\draw[color=blue] (5.1,2.65) ellipse (0.9 and 1.35);
				\node[blue] at (6.3,2.35) {$B_1$};
		
				\filldraw[color=red, fill=red](1.5,6) circle [radius=0.1];
				\node [right] at (1.6,6) {\small $b_i$};
				\filldraw[color=blue, fill=blue](5,6) circle [radius=0.1];
				\node [right] at (5.1,6) {\small $a_i$};

				\draw[color=black, opacity=0.2, very thin] (1.6,6) ellipse (0.96 and 1.9);
				\node[color=black, opacity=0.2, very thin] at (2.9,5.75) {$A_0'$};
				\draw[color=black, opacity=0.2, very thin] (5.1,6) ellipse (0.96 and 1.9);
				\node[color=black, opacity=0.2, very thin] at (6.4,5.75) {$B_0'$};

				\path [fill=red] (1.5,5.8) -- (1.9,4.1) -- (1.2,4.1);    
				\path [fill=blue] (5,5.8) -- (5.8,4.1) -- (4.3,4.1);
				\path [fill=gray, opacity=0.8] (1.6,5.8) -- (4.0,2.75) -- (4.2,3.4);   
				\path [fill=gray, opacity=0.8] (4.9,5.8) -- (2.75,3.3) -- (2.65,3.55);  

				\end{scope}
			\end{tikzpicture}
\caption{The vertices $a_i$ and $b_i$ are swapped if $T_i>0$, where $T_i = e(a_i, B_1) - e(a_i, A_1) + e(b_i, A_1) - e(b_i, B_1)$ : that is, if swapping causes more of the edges between $V_0=A_0 \cup B_0$ and $V_1=A_1\cup B_1$ to lie within the parts. 
}
\label{fig.swapping}
\end{figure}
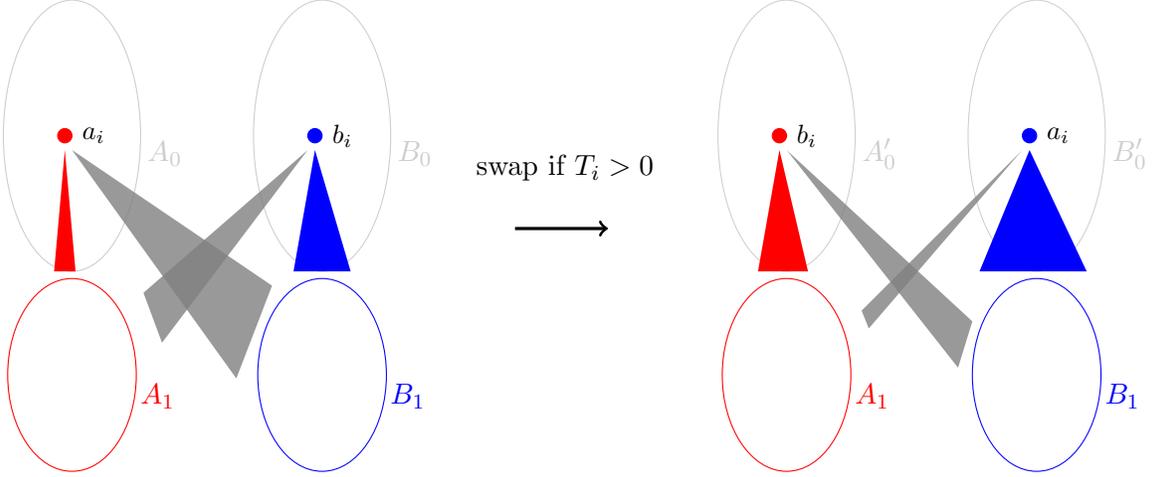

\needspace{2\baselineskip}
\emph{Degree tax}

By Lemma~\ref{lem.swap2}, the random variables $e(a'_i, V_1)$ and $e(b'_i,V_1)$ have the same joint distribution as $e(a_i, V_1)$ and $e(b_i,V_1)$. It follows that the $4k$ random variables $e(a'_i, V_1), e(b'_i,V_1)$ for $i \in [2k]$ are independent, and have the same joint distribution as the $4k$ independent random variables $e(a_i, V_1), e(b_i,V_1)$.
Hence, the joint distribution of $\vol(A')$ and $\vol(B')$ is the same as that
of $\vol(A)$ and $\vol(B)$, and so $q_{\cA'}^D(G) \sim q_{\cA}^D(G)$.
Thus, by Lemma~\ref{lem.orig}
\begin{equation} \label{eqn.degtax1}
q_{\cA}^D(G) = \tfrac12 + O(\tfrac1{n^2}) + o\big(\tfrac{\omega (1-p)}{n^2p}) \;\; \mbox{ whp}. 
\end{equation}
\medskip

\needspace{2\baselineskip}
\emph{$T^*$ is large whp\; (when $np(1-p)$ is large)}

Consider a particular $i \in [2k]$.
Let $0<\eps<1$.  We apply Lemma~\ref{lem.diff} with $n$ replaced by $2k$.  Let $c_0$ be as in Lemma~\ref{lem.diff} for $\eps/2$, and let $c_1=4c_0$.   Assume that $c_1 \leq np \leq n - c_1$, so $c_0 \leq 2kp \leq 2k-c_0$ (for $n$ sufficiently large).  Write $\tilde{n}$ for $6k$ (so $n-5 \leq \tilde{n} \leq n$).  Then, for $n$ sufficiently large,
\begin{equation*} \E[|T_i|] \geq (1- \tfrac{\eps}{2})  \sqrt{ \tfrac{8kp(1-p)}{\pi}} = (1-\tfrac{\eps}{2})  \sqrt{ \tfrac{4\tilde{n}p(1-p)}{3\pi}}. \end{equation*}
Also
\begin{equation*} \var(|T_i|) \leq \E[T_i^2] = \var(T_i) = 4kp(1-p) \leq \tfrac23 np(1-p).\end{equation*}
Hence
\begin{equation*} \E[T^*] \geq (1-\eps/2)  \tfrac{\tilde{n}}3 \sqrt{ \tfrac{4\tilde{n}p(1-p)}{3\pi}} = (1-\eps/2)  \tfrac23 \tfrac1{\sqrt{3\pi}} \sqrt{\tilde{n}^3p(1-p)}\end{equation*}
and
\begin{equation*}\var(T^*) \leq \tfrac29 n^2 p(1-p);\end{equation*}
and so by Chebyshev's inequality
\begin{equation} \label{eqn.tstarbignew}
 T^* \geq \alpha_1 \sqrt{n^3 p(1-p)} \;\; \mbox{ whp},
\end{equation}
where $\alpha_1= (1-\eps) \, \tfrac23 \tfrac1{\sqrt{3\pi}}$.
Note that $\tfrac23 \tfrac1{\sqrt{3\pi}} \approx 0.2171567 > \tfrac15$.
\medskip

\needspace{2\baselineskip}
\emph{Edge contribution \: (when $np(1-p)$ is large)}

Let $p$ be as assumed for~(\ref{eqn.tstarbignew}).
To bound $e(A',B')$, consider separately two sets of possible edges: the $4k^2$ possible edges between $A_0'$ and $B_1$ or $A_1$ and $B_0'$, and the at most $\tfrac14 n^2 - 4k^2$ other possible edges between $A'$ and $B'$.  We have whp
\begin{equation*} \tfrac12 e(V_0,V_1) \leq 4k^2 p + o(\omega \sqrt{n^2 p(1-p)}).\end{equation*}
Thus, by (\ref{eqn.gain}) and~(\ref{eqn.tstarbignew}), whp
\begin{equation*} e(A'_0,B_1)+e(A_1,B'_0)  
\leq 4k^2 p - (\tfrac12 +o(1)) \, \alpha_1 \sqrt{n^3 p(1-p)}. \end{equation*}
Also, whp the number of other edges between $A'$ and $B'$ is at most
\begin{equation*} (\tfrac14 n^2 - 4k^2)p + o(\omega \sqrt{n^2 p(1-p)});\end{equation*}
and so, whp
\begin{equation*} e(A',B') \leq \tfrac14 n^2 p - (\tfrac12 +o(1))\alpha_1 \sqrt{n^3 p(1-p)}. \end{equation*}
Hence by~(\ref{eqn.etot}), 
\begin{equation} \label{eqn.econtribnew}
   q_{\cA'}^{E}(G) = 1 - \frac{e(A',B')}{e(G)} \geq \tfrac12 - \tfrac1{2n} + (1+o(1)) \, \alpha_1 \sqrt{\tfrac{1-p}{np}} \;\; \mbox{ whp}.
\end{equation}
\medskip

\needspace{2\baselineskip}
\emph{Completing the proof of part (a) of Theorem~\ref{thm.lowersqrt}} 

Now we may put together the results~(\ref{eqn.degtax1}) on degree tax and~(\ref{eqn.econtribnew}) on edge contribution. 
With assumptions as for~(\ref{eqn.tstarbignew}) and~(\ref{eqn.econtribnew}),
whp
\begin{equation*} q_{\cA'}(G) \geq \tfrac12 - \tfrac1{2n} +(1 - \tfrac{\eps}{2}) \, \alpha_1 \sqrt{\tfrac{1-p}{np}} - \tfrac12 + O( \tfrac1{n^2})  + o\big(\tfrac{\omega (1-p)}{n^2p}\big) \geq (1 - \eps) \, \alpha_1 \sqrt{\tfrac{1-p}{np}} - \tfrac1{2n}.\end{equation*}
By making $c_0$ larger if necessary, we may ensure that
\begin{equation*}q _{\cA'}(G) \geq (1 - \eps)^2 \, \alpha_1 \sqrt{\tfrac{1-p}{np}} = (1-\eps)^3 \, \tfrac23 \tfrac1{\sqrt{3\pi}} \sqrt{\tfrac{1-p}{np}}  \;\; \mbox{ whp}.
\end{equation*}
This completes the proof of part (a) of the theorem. 
\medskip

\needspace{2\baselineskip}
\emph{Completing the proof of part (b) of Theorem~\ref{thm.lowersqrt}} 

It suffices now to consider $1 \leq np \leq c_0$. Let $X$ and $Y$ be independent, each with distribution $\Bin(2k,p)$, and let $T=X-Y$.  It is easy to see that there is a constant $\delta>0$ such that $\pr(X=0,Y \neq0) \geq \delta$.  Then
\begin{equation*} \E[|T|] \geq 2 \, \pr(X=0,Y \neq0) \geq 2 \delta \geq \alpha_2 \sqrt{np(1-p)},\end{equation*}
 where $\alpha_2= \tfrac{2 \delta}{\sqrt{c_0}}$.  The rest of the proof is as for part (a), with $\alpha_2$ instead of $\alpha_1$.
 \end{proof}

\needspace{6\baselineskip}
\subsection{Constant expected degree case and stochastic block models}
\label{subsec.sb}

The lower bound on modularity in Theorem~\ref{thm.lowersqrt} covers a wide range of probabilities $p$, and has a stand-alone algorithmic proof. Recall that for $1/n \leq p \leq 1- c_0/n$, the algorithm \emph{Swap} whp finds a balanced bipartition achieving modularity at least $\alpha\sqrt{(1\!-\!p)/np}$, where the constant~$\alpha$ may be taken to be $\frac15$ in part of that range. Recent results~\cite{distinguish,banks2016information} on contiguity between Erd\H{o}s-R\' enyi random graphs and stochastic block models allow us to give a better constant for the special case when $p=c/n$.

\begin{thm} \label{thm.lowerSqrt2}
For each constant $c > 1$, we have $\q(\Gncntext) > \frac{0.668}{\sqrt{c}}$ whp. 
\end{thm}
This result may be compared to the value $\q(\Gncntext)\sim 0.97/\sqrt{c}$ predicted using spin-glass models~\cite{trulymodular}. We shall see that, for each $k \geq 2$, whp there is a balanced $k$-part partition with modularity about $f(k)/\sqrt{c}$ for an explicit function~$f(k)>0$, see Table~\ref{fig.numbersH}. 
\begin{prop}\label{prop.plantedLB}
Fix $c > 1$. Whp there is a balanced bipartition $\cA_2$ such that
\begin{equation*}q_{\cA_2}(\Gncntext) \geq \frac{1}{2\sqrt{c}}-o(1);\end{equation*}
and for each $k\geq 3$, whp there is a balanced $k$-part partition $\cA_k$ such that
\begin{equation*}q_{\cA_k}(\Gncntext) \geq \frac{1}{\sqrt{c}}\frac{\sqrt{2(k\!-\!1)\ln(k\!-\!1)}}{k}-o(1).\end{equation*}
\end{prop}

Numerical values are shown in Table~\ref{fig.numbersH}. Choosing $k=6$ parts yields the constant given in Theorem~\ref{thm.lowerSqrt2}, so it suffices now to prove the proposition.

\begin{table}
\centering
$
\begin{array}{c|llllllllllllll}
k 
& 2 & 3 & 4 & 5 & 6 & 7 & 8 & 9 & 10 
\\
\hline
f(k)
&	0.5000              
&   0.5550
&   0.6418
&   0.6660
&   0.6686
&   0.6624
&   0.6524
&   0.6409
&   0.6288
\\
\end{array}
$
\caption{Approximate values of $f(k)$ for small values of $k$. By Proposition~\ref{prop.plantedLB}, whp there is a balanced $k$-part partition of $\Gncntext$ which achieves modularity at least $f(k)/\sqrt{c}- o(1)$, where $f(2)= 1/2$, and $f(k)=\sqrt{2(k\!-\!1)\ln(k\!-\!1)}/k$ for $k\geq 3$.
}\label{fig.numbersH}\end{table}

\begin{proof}
We first consider the case $k=2$, then $k \geq 3$.

\smallskip
\needspace{2\baselineskip}
\emph{Balanced bipartitions}

We consider the planted bisection model essentially as defined in~\cite{distinguish}. Let $n$ be a positive integer, let the vertex set be $V=[n]$, and let $0<\alpha, \beta \leq n$.  Define the random graph $\Rab$ as follows. Let $\sigma_v$ for $v \in V$ be iid random variables, uniformly distributed on $\pm 1$.  Conditional on these labels, each possible edge $uv$ is included with probability $\alpha/n$ if $\sigma_u=\sigma_v$ and with probability $\beta/n$ if $\sigma_u \neq \sigma_v$.  We then ignore the labels.  

When $\alpha$ and $\beta$ are close together, this planted bisection model is \emph{contiguous} to $\Gncntext$ where $c=(\alpha+\beta)/2$; that is, events $A_n$ hold whp in $\Rab$ if and only if they hold whp in $\Gncntext$. The result of Mossel, Neeman and Sly~\cite{distinguish} says precisely that the models are contiguous if and only if $(\alpha-\beta)^2 \leq 2 (\alpha+\beta)$.  It follows that, if we fix $c >1$ and let $\alpha=c+\sqrt{c}$ and $\beta=c-\sqrt{c}$, then the models $\Gncntext$ and $\Rab$ are (just) contiguous. Thus it is sufficient to show that whp we have the claimed bipartition in $\Rab$.

As usual, let $\omega=\omega(n) \to \infty$ as $n \to \infty$, with $\omega=o(n)$.  Let $V_+=\{v : \sigma_v=+1\}$ and $V_-=\{v : \sigma_v=-1\}$, and let $\cA$ be the partition into $V^+$ and $V^-$. We shall use Chebyshev's inequality repeatedly.  For $i=\pm$, $|V_i| = \frac12 n +  o(\sqrt{\omega n})$ whp.  Hence
$e(V_i) = \frac18 \alpha n +  o(\sqrt{\omega n})$ and
$\vol(V_i) =  \frac14 (\alpha+\beta) n +  o(\sqrt{\omega n})$ whp. 
Since $e(\Rab) = \frac14 (\alpha+\beta)n + o(\sqrt{\omega n})$ whp, we have
\begin{equation}\label{eq.twoblocks} q_{\cA}(\Rab) = \frac{\alpha}{\alpha+\beta} - \frac12 + o(\sqrt{\omega/n}) = \frac{\alpha-\beta}{2(\alpha+\beta)} + o(\sqrt{\omega/n}) \;\; \mbox{ whp} ;
\end{equation}
and hence, by our choice of $\alpha$ and $\beta$, whp $q_{\cA}(\Rab) = \frac1{2\sqrt{c}} + o(\sqrt{\omega/n})$.

Further, for any set~$U$ of vertices, let $\isol(U)$ be the number of isolated vertices in~$U$.  Then for $i=\pm$, $\isol(V_i) = \frac12 n e^{-c}+  o(\sqrt{\omega n})$ whp.
We may shuffle isolated vertices in a partition without changing the modularity, so whp we may modify $\cA$ to a balanced partition as required.

\medskip
\needspace{3\baselineskip}
\emph{Balanced $k$-part partitions for $k\geq 3$}

Define the random graph $\Rabk$ by letting $\sigma_v$ for $v\in V$ be iid random variables, uniformly distributed on $[k]$. The possible edges are included at random as for the $\Rab$ model, and we then forget the labels.  By Theorem~1 of~\cite{banks2016information}, for $c=(\alpha+(k-1)\beta)/k$ the models $\Gncntext$ and $\Rabk$ are contiguous if  
$ (\alpha - \beta)^2< 2ck^2 \ln(k\!-\!1)/(k\!-\!1).$  Let $\alpha=c+x\sqrt{c}$ and $\beta=c - (k-1)^{-1} x\sqrt{c}$ where $0<x<\sqrt{ 2(k\!-\!1)\ln(k\!-\!1)}$. (We shall consider $x$ near the upper bound.) Then
\begin{equation*}(\alpha-\beta)^2= \frac{x^2ck^2}{(k\!-\!1)^2} < \frac{2ck^2 \ln(k\!-\!1)}{(k\!-\!1)} \end{equation*}
 and so $\Gncntext$ and $\Rabk$  
are contiguous. Thus it is sufficient to show that whp we have the claimed partition in $\Rabk$.

Let $V_i=\{v : \sigma_v=i\}$ for each $i\in[k]$, and let $\cA$ be the partition into the sets $V_i$.  We again use Chebyshev's inequality repeatedly. For each $i\in[k]$, we have $|V_i| = \frac1k n +  o(\sqrt{\omega n})$ whp.  Hence $e(V_i) = \frac{1}{2k^2} \alpha n +  o(\sqrt{\omega n})$ and
$\vol(V_i)  = \frac{n}{k^2} (\alpha+(k-1)\beta)
=  \frac{1}{k} c n +  o(\sqrt{\omega n})$ whp.
Also, $e(\Rabk) = \frac12 cn + o(\sqrt{\omega n})$ whp; and for each $i$, $\isol(V_i) = \frac1{k} n e^{-c}+  o(\sqrt{\omega n})$ whp.  Hence, whp
\begin{equation*} q_{\cA}(\Rabk) \geq \frac{\alpha /2k}{c/2} - \frac1{k} + o(\sqrt{\omega/ n})
= \frac{x}{k \sqrt{c}} + o(\sqrt{\omega /n}); \end{equation*}
and as before we may shuffle isolated vertices to obtain a balanced partition as required.
\end{proof}

\needspace{12\baselineskip}
\section{Robustness of modularity}
\label{sec.robustness}

This section concerns the robustness of the modularity $\q(G)$ when we change a few edges. We first show in Lemma~\ref{lem.edgeSim} that if we delete a small proportion of edges of a graph, then any change in the modularity is correspondingly small. Note that the modularity can increase or decrease: for example $\, \threeedgepath \,\supset \, \cofourcycler \, \supset \, \cokfourminus \: $ while $\, \q(\threeedgepath)= \frac23 - \frac12= \frac16$, $\, \q(\cofourcycler)= \frac12 \,$ and $\, \q(\cokfourminus)=0$. 

Lemma~\ref{lem.edgeSim} will be used in the proof of Lemma~\ref{lem.upperSqrt}  (which we use to prove Theorem~\ref{thm.usER}(c) and the upper bound in Theorem~\ref{thm.growthRate}) and in Section~\ref{sec.qconc}. 

\begin{lemma}\label{lem.edgeSim}
Let $G=(V,E)$ be a graph, let $E_0$ be a non-empty subset of $E$, let $E'=E \setminus E_0$ and let $G'= (V,E')$.  
Then
\begin{equation}\label{eq.alphabetalem} |\q(G)-\q(G')|<  2 \, |E_0|/|E|.
\end{equation}  
\end{lemma}

\begin{proof}
Clearly we may assume that $G'$ has at least one edge. We shall prove a pair of inequalities which together are slightly stronger than~(\ref{eq.alphabetalem}). Let $\AA$ be any partition of $V$. Let $E_1$ be the set of edges in $E_0$  that lie within parts of $\AA$; and let $E_2= E_0 \setminus E_1$, the set of edges in $E_0$ that lie between parts of $\AA$.  Let  $\alpha=\alpha(\AA)= |E_1|/|E|$ and $\beta=\beta(\AA) = |E_2|/|E|$. Then we make two claims:
we claim that 
\begin{equation}\label{eq.alphabetaVOL}
q_\AA(G')-q_\AA(G)<2\alpha +2\beta,
\end{equation}
and that
\begin{equation}\label{eq.alphabetaDEF}
 \q(G)-\q(G')<2\alpha +\beta. 
 \end{equation}
Let us first show that the claims~\eqref{eq.alphabetaVOL} and \eqref{eq.alphabetaDEF} yield inequality~(\ref{eq.alphabetalem}).
Suppose first that $\q(G')\geq\q(G)$. { Then, by claim~\eqref{eq.alphabetaVOL}, if $\AA$ is an optimal partition for~$G'$,}
\begin{equation*} { |\q(G) - \q(G')|} = q_\AA(G') - \q(G)  \leq q_\AA(G') - q_\AA(G) <  2(\alpha+\beta) = 2|E_0|/|E|,\end{equation*}
so~(\ref{eq.alphabetalem}) holds.
Conversely, if $\q(G)\geq \q(G')$, then~(\ref{eq.alphabetalem}) follows directly from claim~\eqref{eq.alphabetaDEF}. 

It remains to prove the claims~\eqref{eq.alphabetaVOL} and \eqref{eq.alphabetaDEF}.
 For the vertex partition $\AA$ we can calculate the difference in edge contribution between $G$ and $G'$ precisely: we have
\begin{equation*} q^E_\AA(G)=\frac{1}{|E|}\sum_{A\in\AA} e_{G}(A) = \frac{1}{|E|}\Big(\alpha |E|+\sum_{A\in\AA} e_{G'}(A) \Big) 
=\alpha+  (1-\alpha-\beta) q^E_\AA(G').\end{equation*}
Hence 
\begin{equation}\label{eq.removeE} q^E_\AA(G) - q^E_\AA(G') = \alpha - (\alpha+\beta) q^E_\AA(G')
\end{equation}
and in particular 
\begin{equation}\label{eq.removeEb}
q^E_\AA(G')-q^E_\AA(G)\leq \beta.
\end{equation}
Next we bound the possible decrease in degree tax as we move from $G$ to $G'$.
For each part $A_i \in \AA$, let $\alpha_i = |E_1 \cap E(A_i)|/|E|$ and $\beta_i = |E_2 \cap E(A_i, V\backslash A_i)|/|E|$; and note that $\sum_i \alpha_i = \alpha$, $\sum_i \beta_i = 2\beta$ and for each $i$, $\beta_i\leq \beta$.   We may now relate the volumes of any part $A_i$ in $G$ and $G'$: we have
\begin{equation*}\vol_{G'}(A_i) = \vol_G(A_i)-(2\alpha_i+\beta_i)|E|.\end{equation*}
Thus \begin{eqnarray*}
\sum_i \vol_G(A_i)^2 -\sum_i \vol_{G'}(A_i)^2
& = &  \sum_i (\vol_G(A_i) + \vol_{G'}(A_i)) (\vol_G(A_i) - \vol_{G'}(A_i))\\
& < & 2|E| \sum_i \vol_G(A_i)(2\alpha_i+\beta_i).
\end{eqnarray*}
Relabel parts in order of decreasing volume and let $\eta_i = \vol_G(A_i)/(2|E|)$, so $1 \geq \eta_1 \geq \eta_2 \geq \ldots$ and $\sum_i \eta_i = 1$. Observe that $\sum_i \eta_i \alpha_i \leq \eta_1 \alpha \leq \alpha$ and $\sum_i \eta_i \beta_i \leq (\eta_1 + \eta_2) \beta \leq \beta$; so
\begin{equation*}\sum_i \vol_G(A_i)^2 -\sum_i \vol_{G'}(A_i)^2 <  4|E|^2 \sum_i \eta_i (2\alpha_i+\beta_i) \leq 
 4|E|^2( 2\alpha + \beta).\end{equation*}
Since $|E'| \leq |E|$, this gives \begin{equation*}q^D_\AA(G) - q^D_\AA(G') <  2\alpha+\beta;\end{equation*}
and using also~(\ref{eq.removeEb}) we obtain claim \eqref{eq.alphabetaVOL}.\\

\noindent Now we shall prove claim \eqref{eq.alphabetaDEF}.  To bound the possible \emph{increase} in degree tax when we move from $G$ to $G'$, first observe that
\begin{equation*}
\notag \displaystyle q^D_\AA(G) 
>\frac{1}{4|E|^2}\sum_{A\in\AA} \vol_{G'}(A)^2 = (1-\alpha-\beta)^2 q^D_\AA(G')
  \, > (1- 2(\alpha + \beta))\, q^D_\AA(G'),
\end{equation*} 
and so 
\begin{equation*} q^D_\AA(G') - q^D_\AA(G) < 2(\alpha+\beta) q^D_\AA(G').\end{equation*}
Together with~\eqref{eq.removeE} this gives
\begin{eqnarray}\label{eq.removeDX}
q_\AA(G) - q_\AA(G') & < & \alpha - (\alpha+\beta)q^E_\AA(G') +2(\alpha+\beta)q^D_\AA(G').
\end{eqnarray}
 We shall use~\eqref{eq.removeDX} to prove claim~(\ref{eq.alphabetaDEF}), namely that $\q(G)-\q(G') < 2\alpha+\beta$.  Clearly we may assume that $\q(G)>0$. Fix $\AA$ to be an optimal partition for $G$. There are two cases: $q_\AA(G')\geq 0$ and $q_\AA(G')\leq 0$.\\ 

\noindent\textbf{Case 1: $q_\AA(G')\geq 0$.}\\
First note that $\q(G)-\q(G')=q_\AA(G)-\q(G') \leq q_\AA(G)-q_\AA(G')$.
But by~\eqref{eq.removeDX}, and because in this case $q^E_\AA(G')\geq q^D_\AA(G')$, 
\begin{equation*}q_\AA(G) - q_\AA(G')  < \alpha+(\alpha+\beta) \, q^D_\AA(G') \leq 2\alpha+\beta,\end{equation*}
 and claim~(\ref{eq.alphabetaDEF}) follows.

\noindent\textbf{Case 2: $q_\AA(G')\leq 0$.}\\
Since $\q(G')\geq 0$,
\begin{equation*}\q(G)-\q(G')\leq \q(G) = q_\AA(G) = q_\AA(G)-q_\AA(G')+q^E_\AA(G')-q^D_\AA(G').\end{equation*}
Hence by \eqref{eq.removeDX} 
\begin{equation*}
\q(G) -\q(G') <  \alpha + (1-\alpha-\beta)q^E_\AA(G') +(2\alpha+2\beta-1)q^D_\AA(G'),\end{equation*}
and so by re-arranging,
\begin{equation*}\q(G) -\q(G') < \alpha +(\alpha+\beta) q^D_\AA(G') - (1 - \alpha-\beta )\big(q^D_\AA(G') - q^E_\AA(G') \big). \end{equation*} But $\alpha+\beta\leq 1$ and (in this case) $q^D_\AA(G')\geq q^E_\AA(G')$, so 
\begin{equation*}\q(G) -\q(G') < \alpha +(\alpha+\beta)q^D_\AA(G') \leq 2\alpha+\beta. \end{equation*}  
This completes the proof of claim \eqref{eq.alphabetaDEF}, and hence of the lemma.
\end{proof}

There is a similar bound when our two graphs have the same number of edges. The following result will be used in the proof of Theorem~\ref{thm.bconc}, which concerns concentration of modularity in $\Gnp$ and~$\Gnm$. 
\begin{lemma}\label{lem.keepm}
Let $G=(V,E)$ and $G'=(V,E')$ be distinct graphs on the same vertex set $V$, each with $m \geq 1$ edges. Then
{\begin{equation*} |\q(G)-\q(G')| \, < \, \frac{|E \triangle E'|}{m}.\end{equation*}}
\end{lemma}

\begin{proof}
It suffices to consider the case when $|E \triangle E'|=2$. Write $E\triangle E' = \{e, e'\}$ where $e\in E \backslash E'$ and $e'\in E'\backslash E$.  
Let $\cA$ be a partition of $V$, and suppose wlog that $q_\cA(G) \leq q_\cA(G')$.   It suffices to show that 
\begin{equation} \label{eqn.robust}
q_\cA(G') < q_\cA(G) + \frac{2}{m}.
\end{equation}
We consider two cases, depending on whether the edge $e$ is internal or external in $\cA$.

Suppose first that $e$ lies within some part $A$ in $\cA$.  Then $q^E_\cA(G') \leq q^E_\cA(G)$.  Also, if $x=\vol(A)$ then $x \leq 2m$ and so
\begin{equation*}  q^D_\cA(G) - q^D_\cA(G') < \frac{x^2 - (x-2)^2}{4m^2} = \frac{4x-4}{4m^2} < \frac{x}{m^2} \leq  \frac2{m}. \end{equation*}
Thus~(\ref{eqn.robust}) holds in this case. 

Now suppose that $e$ lies between parts $A_1$ and $A_2$ of $\cA$.
Then $q^E_\cA(G') -q^E_\cA(G) \leq \frac1{m}$. 
Let $\vol(A_1)=x_1$ and $\vol(A_2)=x_2$.  Then $x_1+x_2 \leq 2m$, so much as before
\begin{equation*}  q^D_\cA(G) - q^D_\cA(G') < \frac{x_1^2 - (x_1-1)^2 + x_2^2 - (x_2-1)^2}{4m^2} = \frac{2(x_1 +x_2)-2}{4 m^2} < \frac1{m}. \end{equation*}
Hence again~(\ref{eqn.robust}) holds, and we are done.
\end{proof}

The following lemma extends Lemmas~\ref{lem.edgeSim} and~\ref{lem.keepm} (note that if $|E|=|E'|$ then $2|E \backslash E'| = |E \Delta E'|$). 
\begin{lemma}\label{lem.robust4}
Let $G=(V,E)$ and $G'=(V,E')$ be distinct graphs on the same vertex set $V$ with $|E| \geq |E'|$. Then
\begin{equation*} |\q(G)-\q(G')|  \, < \, 
\frac{ 2|E \backslash E'|}{|E|}.\end{equation*}
\end{lemma}
\begin{proof}  
Let $F$ be a set of $|E|-|E'|$ elements of $E \setminus E''$, where $E'' = E \cap E'$.  Let $H$ be the graph on $V$ with edge set $E' \cup F$, with $|E|$ edges.
By Lemma~\ref{lem.edgeSim}, $|\q(H)-\q(G')| \leq 2 |F|/|E|$; and by Lemma~\ref{lem.keepm}, $|\q(G)-\q(H)| \leq  | E \triangle (E' \cup F) |/|E| = 2 (|E'|-|E''|)/|E|$.  Further, either $H \neq G'$ or $H \neq G$ (or both), so at least one of these inequalities is strict.
 Hence
\begin{equation*} |\q(G)\!-\!\q(G')|  <  \frac{2 |F|\!+\!2 (|E'|\!-\!|E''|)}{|E|}
 = \frac{2(|E|\!\!-\!|E''|)}{|E|} 
= \frac{2 |E \backslash E'| }{|E|}, \end{equation*}
as required.
\end{proof}

It is not clear how tight these bounds are.  Let $\alpha^*$ be the infimum of $\alpha$ such that for all distinct graphs $G=(V,E)$ and $G'=(V,E')$ with $|E| \geq |E'|$ we have $|\q(G)-\q(G')| < \alpha |E \backslash E'|/|E|$.
By Lemma~\ref{lem.robust4}
we have $\alpha^* \leq 2$.  The following example shows that $\alpha^* \geq \frac32$.
Let $G$ consist of 3 disjoint edges, with $\q(G)=\frac23$.  By moving one edge we can form $G'$ consisting of a 3-edge path and 2 isolated vertices, with $\q(G') = \frac23 - \frac12 = \frac16$.  Thus 
$|\q(G)-\q(G')|= \frac12 = \frac32 \cdot \frac13$.


\needspace{12\baselineskip}
\section{Upper bounds on modularity}
\label{subsec.ub}

In this section we prove the upper bound on $\q(\Gnp)$ in Theorem~\ref{thm.growthRate}, which also establishes part~(c) 
of Theorem~\ref{thm.usER}.
In Section~\ref{subsec.modProp} we give bounds on the modularity of a graph $G$ in terms of the eigenvalues of its normalised Laplacian $\LL(G)$.
In Section~\ref{subsec.upperSqrt}, these results are used, together with a robustness result from last section and spectral bounds from~\cite{chunglv} and~\cite{amin}, to complete the proof.

\needspace{6\baselineskip}
\subsection{Spectral upper bound on modularity}
\label{subsec.modProp}

The main task of this subsection is prove that the modularity of a graph is bounded above by the spectral gap of the normalised Laplacian. We begin with a definition. Following Chung~\cite{chung1997spectral}, for a graph $G$ on vertex set $[n]$, with adjacency matrix $A_G$ and vertex degrees $d_1,\ldots,d_n>0$, define the \emph{degrees matrix} $D$ to be the diagonal matrix $\mbox{diag}(d_1, \ldots, d_n)$ and the \emph{normalised Laplacian} to be $\mathcal{L}=I-D^{-1/2}A_GD^{-1/2}$.  Here $D^{-1/2}$ is $\mbox{diag}(d_1^{-1/2}, \ldots, d_n^{-1/2})$.
Denote the eigenvalues of $\mathcal{L}$ by $0=\lambda_0\leq \ldots \leq \lambda_{n-1} \,( \leq 2)$, see~\cite{chung1997spectral}.  We call \begin{equation*} \max_{i \neq 0} |1-\lambda_i| = \max \{|1-\lambda_1|, |\lambda_{n-1}-1| \}\end{equation*}
the \emph{spectral gap} of $G$, 
and denote it by $\sg(G)$.
(In terms of the eigenvalues $\tilde{\lambda}_0 \geq \cdots \geq \tilde{\lambda}_{n-1}$ of $D^{-1/2}A_GD^{-1/2}$, we have $\tilde{\lambda}_i = 1- \lambda_i$ and so $\sg(G)= \max_{i \neq 0} | \tilde{\lambda}_i| = \max\{ | \tilde{\lambda}_1|, |\tilde{\lambda}_{n-1}| \}$.)

\begin{lemma}\label{lem.spectralmod} Let $G$ be a graph with at least one edge and no isolated vertices. 
Then
\begin{equation*} q_{\cA}(G) \leq \sg(G) \, (1-1/k) \leq \sg(G)\end{equation*}
for each $k$-part vertex partition $\cA$, and so $\q(G) \leq \sg(G)$.
\end{lemma}

In the special case of $r$-regular graphs, Lemma~\ref{lem.spectralmod} may be written in terms of the spectrum of the adjacency matrix $A_G$, $r = \lambda_0(A_G) \geq \cdots \geq \lambda_{n-1}(A_G)$, since $\sg(G) = \frac{1}{r}\max_{i\neq 0} |\lambda_i(A_G)|$. This special case of Lemma~\ref{lem.spectralmod} is already known, see~\cite{van2010spectral,fasino2014algebraic}, and was used to prove upper bounds on the modularity of random regular graphs in~\cite{treelike} and~\cite{pralat}. 

The proof of Lemma~\ref{lem.spectralmod} relies on a corollary of the Discrepancy Inequality, Theorem 5.4 of~\cite{chung1997spectral}, which is an extension of the Expander-Mixing Lemma to non-regular graphs. Write $\bar{S}=V\backslash S$ where $V=V(G)$.

\begin{lemma}[Corollary 5.5 of~\cite{chung1997spectral}]\label{lem.irrEML}
Let $G$ be a graph with at least one edge and no isolated vertices. Then for each $S\subseteq V$
\begin{equation*} e(S,\bar{S}) \geq (1-\sg(G)) \, \vol(S)\vol(\bar{S})/\vol(G).
 \end{equation*}
\end{lemma}
\begin{proof}[Proof of Lemma~\ref{lem.spectralmod}]
Let $G$ have $m \geq 1$ edges.
Let $\cA=\{A_1,\ldots,A_k\}$ be a vertex partition of $G$. Lemma~\ref{lem.irrEML} guarantees many edges between the parts of $\cA$. The edge contribution satisfies
\begin{equation*} 1-q_\cA^E(G)=\frac{1}{2m}\sum_{i} e(A_i, \bar{A_i}) \geq  (1-\bar{\lambda})\frac{1}{4m^2} \sum_{i} \vol(A_i)\vol(\bar{A_i});\end{equation*}
and
\begin{equation*}
\frac{1}{4m^2}\sum_{i} \vol(A_i)\vol(\bar{A_i}) = \frac{1}{4m^2}\sum_{i} \vol(A_i)(2m-\vol(A_i))  
= 1-q^D_\cA(G).
\end{equation*}
Hence
\begin{equation*}1-q^E_\cA(G)\geq (1-\bar{\lambda})(1-q^D_\cA(G)),\end{equation*} 
and so
\begin{equation*} q_\cA(G)=q^E_\cA(G)-q^D_\cA(G) \leq \bar{\lambda}(1-q^D_\cA(G))\leq \bar{\lambda}(1-\tfrac{1}{k})\end{equation*}
(since $q_{\cA}^D(G) \geq 1/k$ by Lemma~\ref{lem.degtax}).  This completes the proof.
\end{proof}

\needspace{6\baselineskip}
\subsection{The $b (np)^{-1/2}$ upper bound on the modularity $\q(\Gnp)$.}
\label{subsec.upperSqrt}

We are now ready to prove the spectral upper bound for $\q(\Gnp)$.  Let us restate the upper bound in Theorem~\ref{thm.growthRate} as a lemma. (Observe that Lemma~\ref{lem.upperSqrt} implies part (c) of Theorem~\ref{thm.usER}.)
\begin{lemma} \label{lem.upperSqrt}
There is a constant $b$ such that for $\: 0 < p=p(n) \leq 1$ 
\begin{equation*}\q(\Gnp)\leq \frac{b}{\sqrt{np}} \;\;\; \mbox{ whp}.\end{equation*} 
\end{lemma}
\begin{proof}
Notice first that it suffices to show that there exist $c_0$ and $b$ such that for $np \geq c_0$ whp $\q(\Gnp) \leq b/ \sqrt{np}$, and then replace $b$ by $\max\{\sqrt{c_0},b \}$.

For $p \gg \log^2 n /n$, 
the result follows directly from Lemma~\ref{lem.spectralmod}, and Theorem~3.6 of Chung, Vu and Lu \cite{chunglv} (see also (1.2) in~\cite{amin}), which shows that
\begin{equation*} \sg(\Gnp) 
\leq 4(np)^{-1/2}(1+o(1)) \;\; \mbox{whp}. \end{equation*}

For the remainder of the proof we assume that $c_0/n \leq p\leq 0.99$ for some large constant $c_0 \geq 1$. We will use the spectral bound in Lemma~\ref{lem.spectralmod} on a subgraph $H$ which is obtained from the random graph $G= \Gnp$  by deleting a small subset of the vertices (and the incident edges). 

\needspace{3\baselineskip}
Following the construction in~\cite{amin}, let $H$ be the induced subgraph of $G$ obtained as follows.\begin{itemize}
\item Initially set $H= G \setminus \{ v\in V(G) \; : \; d_v<(n-1)p/2\}$.
\item While there is a vertex $v \in V(H)$ with at least 100 neighbours in $V(G) \setminus  V(H)$, remove $v$ from~$H$.
\end{itemize}

Let $V'$ be the set of deleted vertices, and let $E'$ be the set of deleted edges (the edges incident with vertices in $V'$).  Then by Theorem 1.2 of Coja-Oghlan~\cite{amin}, assuming that $c_0$ is sufficiently large, there are positive constants $c_1$ and $c_2$ such that whp  $|V'| \leq n e^{- np/c_2}$ and $\sg(H) \leq c_1 (np)^{-1/2}$.

We want a bound on $|E'|$, not $|V'|$.
By the proof of Corollary~2.3 in~\cite{amin}, whp in $\Gnp$ we have
$\vol(S) \leq 2np|S|+n e^{-np/1500}$ simultaneously for each set $S$ of vertices.  (The result is stated with $\vol(S)$ replaced by $|N_G(S)|$, the number of neighbours of $S$ outside $S$, but the proof actually shows the result for $\vol(S)$.)
Hence, noting also that $np \geq 1$ and setting $c_3= \max\{c_2, 1500\}$, whp 
\begin{equation*} |E'|\leq \vol(V') \leq 2n^2 p \,e^{-np/c_2} + ne^{-np/1500} \leq 3 n^2 p \, e^{-np/c_3} \leq e(G) \cdot 9 e^{-np/c_3},\end{equation*} 
where the last inequality follows since whp $e(G) \geq n^2p/3$.  By making $c_0$ larger if necessary we can ensure that $9 e^{-np/c_3} \leq \frac13 \, (np)^{-1/2}$. Now, by Lemma~\ref{lem.spectralmod}, whp
\begin{equation*}  \q(G \setminus E') = \q(H) \leq \sg(H) \leq c_1 (np)^{-1/2}. \end{equation*}
Hence, by Lemma~\ref{lem.edgeSim}, whp
\begin{equation*} \q(G) \leq \q(G \setminus E') + 2|E'|/e(G) \leq (c_1+ 1) \, (np)^{-1/2},\end{equation*}
and the proof is complete.
\end{proof}

\begin{rem}\label{rem.distinguish} 
\normalfont The upper bound on $\q(\Gnp)$ just proven implies that modularity values will whp distinguish the stochastic block model from the Erd\H{o}s-R\'enyi model, when the probabilities are only a constant factor past the detectability threshold, as we now explain. Consider the stochastic block model in which there is a hidden partition of the vertex set into two parts $V_{-}$ and $V_{+}$, and the edges are placed with probability $p$ inside these parts and with probability $p'<p$ between the parts -- see Section~\ref{subsec.sb} for the definition. It is a challenge to distinguish this model from the Erd\H{o}s-R\'enyi random graph with the same expected edge density, and there are theoretical limits on how close the edge probabilities $p$ and $p'$ can be for this to be possible~\cite{distinguish}.

Suppose that the planted partition has edge probabilities $p=\alpha/n$ (inside parts) and $p'=\beta/n$ (between parts) for constants $\alpha > \beta$; and denote the random graph by $\Rab$. As shown earlier, see~\eqref{eq.twoblocks}, the modularity score of the planted partition itself is whp $(\alpha\!-\!\beta)/2(\alpha\!+\!\beta)+o(1)$.
By Theorem~\ref{thm.growthRate} there is a constant~$b$ such that, for the Erd\H{o}s-R\'enyi random graph with edge probability $(\alpha\!+\!\beta)/2n$, the maximum modularity is whp less than $b \, \sqrt{2/(\alpha\!+\!\beta)}$. Thus for $(\alpha\!-\!\beta)^2 > 8b^2(\alpha\!+\!\beta)$, whp the modularity score of the planted partition in $\Rab$ is higher than that of any partition of the Erd\H{o}s-R\'enyi random graph $G_{n,(\alpha+\beta)/2n}$. In particular, for $(\alpha\!-\!\beta)^2 > 8b^2(\alpha\!+\!\beta)$, if the procedure is to flip a coin and sample the stochastic block model $\Rab$ if heads and Erd\H{o}s-R\'enyi random graph $G_{n,(\alpha\!+\!\beta)/2n}$ if tails, then whp the modularity of the random graph would tell us the outcome of the coin flip, i.e. it distinguishes the stochastic block model from the Erd\H{o}s-R\'enyi model.

The theoretical lower bound for detectability is $(\alpha\!-\!\beta)^2 \geq 2(\alpha\!+\!\beta)$~\cite{distinguish}. Thus at a constant factor, namely $4b^2$, past the detectability threshold, modularity whp distinguishes the stochastic block model from the Erd\H{o}s-R\'enyi model.
\end{rem}


\needspace{12\baselineskip}
\section{Concentration and expectation of $\q(\Gnp)$ }
\label{sec.qconc}

We shall see in Theorem~\ref{thm.bconc} that the modularity of our random graphs  
is highly concentrated about the expected value. We use the result for $\q(G_{n,m})$ to deduce that for $\q(\Gnp)$.

\begin{thm}\label{thm.bconc}
(a) Given $n \geq 1$ and $0 \leq m \leq \binom{n}{2}$, for each $t>0$ 
\begin{equation*} \pr\Big( \big| \q(\Gnm) - \E[\q(\Gnm)] \big| \geq t \Big) < 2e^{-t^2m/2}. \end{equation*}

(b) There is a constant $\eta>0$ such that for each $n \geq 1$ and each $0<p<1$ the following holds, with $\mu=\mu(n,p) = \binom{n}{2} p$.  For each $t \geq 0$ 
\begin{equation*} \pr\Big( \big| \, \q(\Gnp) - \E[\q(\Gnp)] \, \big| \geq t \Big) < 2 \, e^{-\eta \mu t^2}.\end{equation*}
\end{thm}
For example, we may use part (b) to consider separately small and large deviations for $\q(\Gnp)$. 

\begin{cor}\label{cor.bconc}
Let $c>0$ be a constant, let $p=p(n)$ satisfy $np \geq c$.
Then  the variance of $\q(\Gnp)$ is $O(1/n)$, so for any function $\omega(n)\rightarrow \infty$
\begin{equation*} \left| \q(\Gnp)-\E[\q(\Gnp)]\right| \leq \frac{\omega(n)}{\sqrt{n}} \;\; \mbox{whp};\end{equation*}
and, for any fixed $\eps>0$, 
\begin{equation*} \pr \left( \left| \q(\Gnp)-\E[\q(\Gnp)] \right|\geq \eps \right)
=e^{-\Omega(n) }.\end{equation*}
\end{cor}
The only part of Corollary~\ref{cor.bconc} that is not immediate is to check that the variance is as claimed.  Let $n \geq 3$, and let $X= \q(\Gnp)-\E[\q(\Gnp)]$.  Then, for each $t>0$, by Theorem~\ref{thm.bconc} (b) (and noting that 
$\mu/n \geq \tfrac12 c(n\!-\!1)/n \geq \tfrac13 c$),
\begin{equation*} \pr(nX^2 \geq t) = \pr(|X| \geq \sqrt{t/n}) \leq 2 \, e^{-\eta \mu t/n} \leq 2 e^{-\tfrac13 \eta c \, t}.\end{equation*}
It follows that $\E[nX^2]$ is at most some constant $\alpha$, and so $\var(\q(\Gnp)) \leq \alpha/n$, as required.

To prove Theorem~\ref{thm.bconc} we make use of Lemmas~\ref{lem.edgeSim} and~\ref{lem.keepm} which bound the sensitivity of modularity to changes in the edge set. We also use the following concentration result from~\cite{bbddiffs} Theorem 7.4 (see also Example 7.3) or Theorem~3.3 of~\cite{mckay2016degree}.

\needspace{4\baselineskip}
\begin{lemma}
\label{lem.setconc}
 Let $A$ be a finite set, let $a$ be an integer such that $0\leq a \leq |A|$, and consider the set $\binom{A}{a}$ of all $a$-element subsets of $A$.
  Suppose that the function $f:\binom{A}{a}\rightarrow \R$ satisfies $|f(S)-f(T)|\leq c$ whenever $|S \triangle T|=2$  (i.e.\ the $a$-element subsets $S$ and $T$ are minimally different). If the  random variable $X$ is uniformly distributed over $\binom{A}{a}$, then 
 \begin{equation*} \pr \left( \Abs{f(X)-\E[f(X)]} \ge t \right)  \le 
  2 e^{-2t^2/a c^2}.\end{equation*} 
\end{lemma}

Recall from Lemma~\ref{lem.keepm} that if $E(G)$ and $E(G')$ are both of size $m$ and are minimally different then $|\q(G)-\q(G')| < 2/m$. Hence, Lemma~\ref{lem.setconc} with $a=m$ and $c=2/m$ immediately yields 
part (a) of Theorem~\ref{thm.bconc}.

\begin{proof}[Proof of Theorem~\ref{thm.bconc} part (b)] 
Let $G \sim \Gnp$.  Let $M=e(G)$ and let $\mu=\mu(n,p) = \E[M] = \binom{n}{2} p$. 
We will first show the more detailed statement that,
for each $t \geq 42/\sqrt{\mu}$, we have 
\begin{equation} \label{eqn.showconc}
\pr \big( \big|\q(G)-\E[\q(G)] \big| \geq t \big) \leq 5 e^{- t^2 \mu /103},
\end{equation}
from which we will deduce part (b) of the theorem easily. Clearly we may assume that $0 \leq t \leq 1$. Define the event $\mathcal{E} = \{ M > 2\mu /3 \}$. Now, letting $\cE^c$ denote the complement of $\cE$,
\begin{eqnarray}\label{eq.splitintotwo}
&&\!\!\!\!\!\!\!\!\pr \big( \big|\q(G)\!-\!\E[\q(G)] \big| \geq t \big) \\
\notag&&\leq 
\pr\big(\big(\big|\q(G)\!-\!\E[\q(G)|M] \big| \geq \tfrac{t}{2} \big)  \land \cE \big) + \pr\big( \big|\E[\q(G)|M]\!-\!\E[\q(G)] \big|\geq \tfrac{t}{2}  \big) \land \cE \big) + \pr(\cE^c).
\end{eqnarray}
The proof proceeds by bounding separately the terms on the right in \eqref{eq.splitintotwo}.

Firstly, by using part (a) of the theorem and conditioning on $M=m$ where $m >2\mu/3$, we have
\begin{eqnarray}
\label{eq.firstterm}
\pr\big( \big(\, \big|\, \q(G)-\E[\q(G)|M] \,\big| \geq \tfrac{t}{2} \big) \land \cE \Big)
&\leq& 
2 \exp ( - \tfrac12 (\tfrac{t}{2})^2 (\tfrac{2 \mu}{3}) ) = 2\exp(-\tfrac{t^2 \mu}{12} ).
\end{eqnarray}
We now work towards a bound of the second term of~\eqref{eq.splitintotwo}. Let $G'\sim \Gnp$  independently of $G$, and let $M'=e(G')$. By Lemma~\ref{lem.edgeSim} and a simple coupling argument, for $0< m \leq \binom{n}{2}$
\begin{equation*} \Big| \E\big[\q(G)|M=m] - \E[\q(G')|M'=m'] \Big| \leq \frac{2|m-m'|}{\max\{m,m'\}} \leq \frac{2|m-m'|}{m} .\end{equation*} 
Also, for any $x$ and any random variable $Y$ (with finite mean) we have
$\big| x - \E[Y] \big| \leq \E_Y[|x-Y|]$.  
Thus, for $0<m \leq \binom{n}{2}$, 
\begin{eqnarray*}
\big|\E [\q(G)|M=m] -\E(\q(G'))\big| & \leq  & 
\E_{M'}\big[ \big|\E [\q(G)|M=m]-\E[\q(G')|M'] \big| \big] \\
& \leq & (2/m) \, \E_{M'} [|m-M'|] \\
& \leq &   
(2/m)\, \big(|m-\mu|+ \E[|M'-\mu|] \big)\\
& \leq &   
(2/m)\, \big(|m-\mu|+ \sqrt{\mu} \big),
\end{eqnarray*}
since 
$\E[|M'-\mu|]  \leq \sqrt{\E[(M' -\mu)^2]} 
 \leq \sqrt{\mu}$.  Hence
\begin{eqnarray*}
\pr\Big( \big(\big|\E[\q(G)|M] - \E[\q(G)] \big| \geq \tfrac{t}{2} \big) \land \cE \Big) 
\notag
& \leq & \pr\Big( \big( | M-\mu| +\sqrt{\mu} \geq \tfrac{tM}{4} \big) \land \cE \Big)\notag \\
&\leq& 
\pr\Big( | M-\mu| \geq \tfrac{t \mu}{6}-\sqrt{\mu}  \Big) \notag \\ 
&\leq& 
\pr\Big( | M-\mu| \geq \tfrac{t \mu}{7}  \Big) 
\end{eqnarray*} 
since we assumed that $t \geq 42\mu^{-1/2}$ and so $t\mu/6-\sqrt{\mu} \geq t\mu/7$.
But by a Chernoff inequality (see for example Theorem 2.1 of~\cite{jbook})
\begin{equation*} \pr\Big( | M-\mu| \geq \frac{t \mu}{7}  \Big)  \leq 2\exp\Big( -\frac{(t/7)^2 \mu}{2(1+ \tfrac13(t/7))}  \Big) 
\;\; \leq \;\; 2\exp\Big( -\frac{t^2\mu}{103}  \Big),
\end{equation*}
so
\begin{equation} \label{eqn.new}
 \pr\Big( \big(\big|\E[\q(G)|M] - \E[\q(G)] \big| \geq \tfrac{t}{2} \big) \land \cE \Big) \leq 2\exp\Big( -\frac{t^2\mu}{103}  \Big).
\end{equation}

Finally, by a Chernoff inequality (again see Theorem 2.1 of~\cite{jbook}),
\begin{equation*}\pr(\cE^c) \leq \exp( - \tfrac12 (\tfrac 13)^2 \mu) = \exp(-\mu/18) 
\leq \exp(-t^2\mu/18). \end{equation*}
This inequality, together with~(\ref{eq.splitintotwo}), \eqref{eq.firstterm} and~(\ref{eqn.new}), yields~(\ref{eqn.showconc}).
It remains to use~(\ref{eqn.showconc}) to deduce the statement in part (b) of the theorem. 
Let $a=e^{42^2/103} ( \geq 5)$.  Then $a e^{- t^2 \mu /103} \geq 1$ for $0 \leq t \leq 42/\sqrt{\mu}$; and so by~(\ref{eqn.showconc}) 
\begin{equation*} \pr \big( \big|\q(G)-\E[\q(G)] \big| \geq t \big) \leq a e^{- t^2 \mu /103} \; \mbox{ for all } t \geq 0.\end{equation*}
Now let $\beta = \log_2 a$, so $\beta >1$ and $2(1/a)^{1/\beta} = 1$.  Thus $2 y^{1/\beta} \geq 1$ for each $y \geq 1/a$.  Also, letting $f(y)= 2y^{1/\beta} - a y$ for $y \geq 0$, we have $f(y) \geq 0$ for each $0 \leq y \leq 1/a$, since $f(0)= f(1/a) = 0$ and $f$ is increasing then decreasing.  Therefore 
\begin{equation*} \min \{1, ay \} \leq 2y^{1/\beta} \;\; \mbox { for all } y \geq 0. \end{equation*}
Hence, letting $\eta = 1/(103\beta)$, we have
\begin{equation*} \pr \big( \big|\q(G)-\E[\q(G)] \big| \geq t \big) \leq 2 e^{- \eta \mu t^2} \; \mbox{ for all } t \geq 0,\end{equation*}
as required.
\end{proof}

We may use the first robustness lemma, Lemma~\ref{lem.edgeSim}, to show that the expected modularity of a random graph with edge probability~$p$ is similar to that of a random graph with edge probability~$p'$ when~$p'$ is near~$p$ (and $n$ is large). At the moment it is an open question whether the expected modularity $\E(\q(G_{n,c/n}))$ tends to a limit~$f(c)$ as $n\rightarrow \infty$. However, if such a limit $f(c)$ did exist then Lemma~\ref{lem.qcts} would show that $f(c)$ is uniformly continuous in~$c$. \\

\needspace{4\baselineskip}
\begin{lemma}\label{lem.qcts}
Let $\eps>0$.  If $n^2p 
\rightarrow \infty$ and $p\leq p'\leq (1+\tfrac{\eps}3) \, p$ then for $n$ sufficiently large, \begin{equation*}  \Big| \E [\q(\Gnp)] - \E [\q(\Gnpdash)] \Big| < \eps. \end{equation*}
\end{lemma}
\begin{proof}
First let us consider very small $p$ and large $p$. If $p\leq 1/n$ then whp $\q(\Gnp)> 1-\eps/2$ by Theorem~\ref{thm.usER}(a), and so for large enough $n$ we have $\E[\q(\Gnp)]>1-\eps$.  
On the other hand, for large $p$, by Theorem~\ref{thm.growthRate} there exists a constant $K=K(\eps)$ so that for $p\geq K/n$ whp $\q(\Gnp) < \eps/2$ and so for large enough $n$, $\E[\q(\Gnp)]<\eps$. Hence we may assume that $p,p'=\Theta(1/n)$.

To sample $\Gnpdash$ we may first sample edges with probability $p$ and then independently resample with probability $p''=(p'-p)/(1-p)$. Write~$\GG_n$ for the set of all graphs on vertex set $[n]$. For $H, H'\in \GG_n$ write $H\cup H'$ to denote the (simple) graph with vertex set $[n]$ and edge set $E(H)\cup E(H')$. 
Then
\begin{equation}\label{eq.doubleexposure} \E\big(\q(G_{n,p}) -\q(G_{n,p'}) \big)= \!\!\sum_{H\in \GG_n}\!\! \pr(G_{n,p}=H) \!\!\sum_{H'\in \GG_n} \!\!\pr(G_{n,p''}=H')(\q(H)-\q(H\cup H')).
\end{equation}
Let $p^+= \tfrac{\eps}{3} \tfrac{p}{1-p}$, and note that $p'' \leq p^+$.
Let $J$ be the event that $e(G_{n,p})\geq \tfrac{19}{20} \cdot\binom{n}{2}p $ and $e(\Gnpdashdash)\leq \tfrac{21}{20} \cdot \binom{n}{2}p^+$; and notice that, for large enough $n$, event $J$ occurs with probability at least $1-\eps/6$.
 If $J$ holds, and $n$ is sufficiently large that $1-p > \tfrac{18}{19}$, then the number of edges added in the second exposure is a small proportion of those already there:
$e(\Gnpdashdash) < \tfrac7{18} \eps \, e(\Gnp)$
 and we can apply Lemma~\ref{lem.edgeSim}. Hence by~\eqref{eq.doubleexposure}
\begin{equation*}|\E\big[\q(\Gnp) -\q(\Gnpdash) \big]| <
\tfrac79 \eps + \pr(J^c) < \eps,\end{equation*}
which completes the proof.
\end{proof}

\needspace{12\baselineskip}
\section{Concluding remarks}
\label{sec.concl}

In this section, we briefly describe what we have done in this paper; mention some other current work; and present two questions, one concerning modularity just above the threshold, and one question inspired by the statistical physics literature concerning partitions with few parts.

The definition of modularity is most well-fitted to graphs that are reasonably sparse.
We have given quite a full picture of the behavior of the modularity of the  
random graphs $\Gnp$ and $\Gnm$, for a wide range of densities. We have not looked in detail here inside the critical window, when the giant component is forming.  Also, we have not looked in detail here at the very dense case: that is done in the companion paper~\cite{vdense}, which in particular investigates the threshold when the modularity drops to exactly 0, and finds that this happens when the complementary graph has average degree~1. Another companion paper~\cite{extreme} investigates the maximum and minimum modularity of graphs with given numbers of edges or given density.  These results help to set in context the results given here on the modularity of random graphs.
A further related paper `Modularity and edge-sampling'~\cite{sampling} considers the situation where there is an unknown underlying graph $G$ on a large vertex set, and we can test only a proportion $p$ of the possible edges to check whether they are present in $G$. It investigates how large $p$ should be so that
the modularity of the observed graph $G'$ is likely to give good upper or lower bounds on $\q(G)$. 

We refer the reader also to~\cite{treelike} for open questions on regular graphs. For any 3-regular graph the modularity is at least $2/3\!-\!o(1)$ and the modularity of a random 3-regular graph $\q(G_{n,3})$ is whp in the range $(0.66,0.81)$ see~\cite{treelike, pralat}. Is it possible to bound $\q(G_{n,3})$ strictly above $2/3+\eps$ whp? Is it the case that of  3-regular graphs the random graph has asymptotically the lowest modularity?

\smallskip 
\needspace{3\baselineskip}
\emph{Modularity just above the threshold}

In Theorem~\ref{thm.sparse} part (iii) and comments at the end of Section~\ref{sec.mid}, we learned about the behaviour of $\q(\Gnp)$ when $np$ is just above the threshold value 1.  We saw that if $np=1+t$ with $t>0$ and small, then whp the modularity deficit $1-\q(\Gnp)$ has order between $t^2$ and about $t^3$. It would be interesting to learn more about this  threshold behaviour. 

\begin{conj}\label{conj.squaret} If $np=1+t$ with $t>0$ sufficiently small, then whp 
\begin{equation*} 1-\q(\Gnp) = \Theta(t^2).\end{equation*}
\end{conj}
In full detail, the conjecture is that there exist $t_0>0$ and $0<a<b$ such that for each fixed $0<t \leq t_0$, whp $at^2 \leq 1-\q(G_{n,(1+t)/n}) \leq bt^2$. 

For the connected components partition $\cal C$, by Lemma~\ref{lem.qCC} there exists $t_0>0$ such that for each $0<t \leq t_0$, whp 
\begin{equation*}  15t^2 < 1 - q_\CC(\Gnp) < 16 t^2.\end{equation*}
Thus Conjecture~\ref{conj.squaret} would imply that whp $\cal C$ has modularity deficit (that is, 1 minus its modularity score) of optimal order in terms of $t$, although by~\cite{thesis} we know that whp $\cal C$ is not the optimal partition in this range. 
\smallskip
 
\needspace{3\baselineskip}
\noindent
\emph{Do few parts suffice?}

Corollary~\ref{cor.growthRate} confirmed the $c^{-1/2}$ growth rate conjectured for the modularity of $\Gncntext$ by Reichardt and Bornholdt~\cite{trulymodular}. In that paper, it was also conjectured that the optimal partition would have five parts. This is not exactly true, since every optimal partition for a graph must have at least as many parts as there are connected components of size at least 2, and whp there are linearly many isolated edges in $\Gncntext$. However, an approximate version of the prediction may be correct: perhaps whp there is a partition with only five parts which has modularity score close to the optimum. Let us explore further.
 
Given a graph $G$ and a positive integer $k$,  let $q_{\leq k}(G)$ be the maximum modularity score of a vertex partition with at most $k$ parts; that is, $q_{\leq k}(G) = \max_{|\cA| \leq k} q_\cA(G)$. By Lemma~1 of~\cite{dinh2011finding}, for every graph $G$ and positive integer $k$, 
\begin{equation} \label{eqn.qkdiff1}
 q_{\leq k}(G) \geq \q(G) \, (1-1/k).
 \end{equation}
On the other hand, we shall see that for every $c>0$ there is a constant $\delta=\delta(c)>0$, such that for each positive integer $k$

\begin{equation} \label{eqn.qkdiff2}
 q_{\leq k}(\Gncntext) \leq \q(\Gncntext) \, (1 -  \delta/k) \;\; \mbox{ whp}.
\end{equation}
To prove this, let $G$ be any graph with at least one edge, and let $\cA=(A_1,\ldots,A_{k'})$ (where $k' \leq k$) be a partition achieving the optimal modularity score over all partitions with at most $k$ parts.  Suppose that $G$ has components $C_1,C_2,\ldots$ ordered by decreasing volume.  We focus on $C_1$. 
Let $a_j= \vol(A_j \cap V(C_1))$ and $b_j = \vol(A_j \backslash V(C_1))$ for $j=1,\ldots,k'$.  Then
\begin{equation*} \sum_j ((a_j+b_j)^2 - a_j^2) = \sum_j (2a_j b_j + b_j^2) > \sum_j b_j^2 \geq (\sum_j b_j)^2/k \, = \tfrac1{k} \vol^2(G \, \backslash \, C_1) .\end{equation*}
Hence
\begin{equation*} \vol^2(G) \cdot q^D_{\cA}(G) = \sum_j (a_j+b_j)^2 \geq \sum_j a_j^2 + \tfrac1{k} \vol^2(G \, \backslash \, C_1). \end{equation*} 
Let $\cB$ be the vertex partition with parts the non-empty sets $A_i \cap V(C_1)$ together with the parts $V(C_2)$, $V(C_3)$, $\ldots$ and 
 observe that $q_{\cB}^E(G) \geq q_{\cA}^E(G)$. Thus, 
\begin{equation*} \vol^2(G) \cdot q^D_{\cB}(G) = \sum_j a_j^2 + \sum_{i \geq 2} \vol^2(C_i). \end{equation*} 
Hence
\begin{equation*} \q(G) - q_{\leq k}(G) \geq q_{\cB}(G) - q_{\cA}(G)  \geq  q_{\cA}^D(G) -  q_{\cB}^D(G) \geq \frac{\vol^2(G \, \backslash \, C_1)}{k \, \vol^2(G)} - \frac{\sum_{i  \geq 2} \vol^2(C_i)}{\vol^2(G)}.\end{equation*}
But when $G \sim \Gncntext$, the second term tends to 0 in probability; and $\vol(G \, \backslash \, C_1) / \vol(G)$ tends in probability to a constant $y = y(c)>0$  (where $y=1-x^2/c^2$ in the notation in Lemma~\ref{lem.qCC}) so the first term above tends in probability to $y^2/k$.
Thus, if we let $\delta= \frac12 y^2 >0$, then
\begin{equation*}
q_{\leq k}(\Gncntext) \leq \q(\Gncntext) -  \delta/k \leq \q(\Gncntext) \, (1 -  \delta/k)  \;\; \mbox{ whp}, \end{equation*}
and we have proved~(\ref{eqn.qkdiff2}).

In the spirit of~(\ref{eqn.qkdiff1}), and despite~(\ref{eqn.qkdiff2}), we propose the following amended version of the `five parts'  conjecture of Reichardt and Bornholdt~\cite{trulymodular}.
\begin{conj}\label{conj.onlykparts}
There exist a positive integer $k$ with the property that, for each $\eps>0$ there exists $c_0$ such that, if $c \geq c_0$ then 
\begin{equation*} q_{\leq k}(\Gncntext) \geq \q(\Gncntext) \, (1-\eps) \;\; \mbox{ whp}.\end{equation*}
\end{conj}

Observe that, by~(\ref{eqn.qkdiff1}), this inequality must hold if $\eps \geq 1/k$. We are conjecturing that there is some finite number $k$ of parts (perhaps $k=5$?) such that whp an optimal partition over the restricted class with at most $k$ parts achieves modularity score at least $(1-\eps(c))$ times the (unrestricted) optimal value, where $\eps(c) \to 0$ as $c \to \infty$.

\needspace{12\baselineskip}

\bibliographystyle{myplain}
\bibliography{articles}

\begin{thebibliography}{10}

\bibitem{chunglupowerlaw}
W.\ Aiello, F.\ Chung, and L.\ Lu.
\newblock A random graph model for power law graphs.
\newblock {\em Experimental Mathematics}, 10(1):53--66, 2001.

\bibitem{bagrow}
J.~P.\ Bagrow.
\newblock Communities and bottlenecks: Trees and treelike networks have high
  modularity.
\newblock {\em Physical Review E}, 85(6):066118, 2012.

\bibitem{banks2016information}
J.\ Banks, C.\ Moore, J.\ Neeman, and P.\ Netrapalli.
\newblock Information-theoretic thresholds for community detection in sparse
  networks.
\newblock In {\em Conference on Learning Theory}, pages 383--416, 2016.

\bibitem{bolla2015spectral}
M.\ Bolla, B.\ Bullins, S.\ Chaturapruek, S.\ Chen, and K.\ Friedl.
\newblock Spectral properties of modularity matrices.
\newblock {\em Linear Algebra and Its Applications}, 473:359--376, 2015.

\bibitem{clusterduck}
U.\ Brandes, D.\ Delling, M.\ Gaertler, R.\ G{\"o}rke, M.\ Hoefer, Z.\
  Nikoloski, and D.\ Wagner.
\newblock On finding graph clusterings with maximum modularity.
\newblock In {\em Graph-Theoretic Concepts in Computer Science}, pages
  121--132. Springer, 2007.

\bibitem{nphard}
U.\ Brandes, D.\ Delling, M.\ Gaertler, R.\ Gorke, M.\ Hoefer, Z.\ Nikoloski,
  and D.\ Wagner.
\newblock On modularity clustering.
\newblock {\em Knowledge and Data Engineering, IEEE Transactions on},
  20(2):172--188, 2008.

\bibitem{chung1997spectral}
F.\ Chung.
\newblock {\em Spectral graph theory}, volume~92.
\newblock American Mathematical Soc.\, Providence, RI, 1997.

\bibitem{chunglv}
F.\ Chung, L.\ Lu, and V.\ Vu.
\newblock The spectra of random graphs with given expected degrees.
\newblock {\em Internet {M}athematics}, 1(3):257--275, 2003.

\bibitem{amin}
A.\ Coja-Oghlan.
\newblock On the {L}aplacian {E}igenvalues of ${G}_{n,p}$.
\newblock {\em Combinatorics, Probability and Computing}, 16:923--946, 2007.

\bibitem{modgraphclasses}
F.\ De~Montgolfier, M.\ Soto, and L.\ Viennot.
\newblock Asymptotic modularity of some graph classes.
\newblock In {\em Algorithms and Computation}, pages 435--444. Springer, 2011.

\bibitem{dinh2015network}
T.~N.\ Dinh, X.\ Li, and M.~T.\ Thai.
\newblock Network clustering via maximizing modularity: Approximation
  algorithms and theoretical limits.
\newblock In {\em Data Mining (ICDM), 2015 IEEE International Conference on},
  pages 101--110. IEEE, 2015.

\bibitem{dinh2011finding}
T.~N.\ Dinh and M.~T.\ Thai.
\newblock Finding community structure with performance guarantees in scale-free
  networks.
\newblock In {\em Privacy, Security, Risk and Trust (PASSAT) and 2011 IEEE
  Third Inernational Conference on Social Computing (SocialCom), 2011 IEEE
  Third International Conference on}, pages 888--891. IEEE, 2011.

\bibitem{beatejournal}
B.\ Ehrhardt and P.~J.\ Wolfe.
\newblock Network modularity in the presence of covariates.
\newblock {\em Siam Review}, 61(2):261--276, 2019.

\bibitem{ERgiant}
P.\ Erd{\H{o}}s and A.\ R\'enyi.
\newblock On the evolution of random graphs.
\newblock {\em Publications of the Mathematical Institute of the Hungarian
  Academy of Sciences}, 5:17--61, 1960.

\bibitem{fasino2014algebraic}
D.\ Fasino and F.\ Tudisco.
\newblock An algebraic analysis of the graph modularity.
\newblock {\em SIAM Journal on Matrix Analysis and Applications},
  35(3):997--1018, 2014.

\bibitem{fortunato2010community}
S.\ Fortunato.
\newblock Community detection in graphs.
\newblock {\em Physics Reports}, 486(3):75--174, 2010.

\bibitem{FortBart2008}
S.\ Fortunato and M.\ Barth\'elemy.
\newblock Resolution limit in community detection.
\newblock {\em Proceedings of the National Academy of Sciences}, 104(1):36--41,
  2007.

\bibitem{frieze2015book}
A.\ Frieze and M.\ Karo{\'n}ski.
\newblock {\em Introduction to random graphs}.
\newblock Cambridge University Press, 2015.

\bibitem{GPA04}
R.\ Guimer{\`a}, M.\ Sales-Pardo, and L.~A.~N.\ Amaral.
\newblock Modularity from fluctuations in random graphs and complex networks.
\newblock {\em Physical Review E}, 70:025101, 2004.

\bibitem{guimera2007module}
R.\ Guimer{\`a}, M.\ Sales-Pardo, and L.~A.~N.\ Amaral.
\newblock Module identification in bipartite and directed networks.
\newblock {\em Physical Review E}, 76(3):036102, 2007.

\bibitem{JansonSuss}
S.\ Janson and M.~J.\ Luczak.
\newblock Susceptibility in subcritical random graphs.
\newblock {\em Journal of Mathematical Physics}, 49(12):125207, 2008.

\bibitem{jbook}
S.\ Janson, T.\ {\L}uczak, and A.\ Ruci{\'n}ski.
\newblock {\em Random {G}raphs}, volume~45.
\newblock John Wiley \& Sons, 2011.

\bibitem{kaminski2018clustering}
B.\ Kaminski, V.\ Poulin, P.\ Pralat, P.\ Szufel, and F.\ Theberge.
\newblock Clustering via hypergraph modularity.
\newblock {\em arXiv preprint arXiv:1810.04816}, 2018.

\bibitem{kanter1987graph}
I.\ Kanter and H.\ Sompolinsky.
\newblock Graph optimisation problems and the {P}otts glass.
\newblock {\em Journal of Physics A}, 20(11):L673, 1987.

\bibitem{hyperbolicIntro}
D.\ Krioukov, F.\ Papadopoulos, M.\ Kitsak, A.\ Vahdat, and M.\ Bogun{\'a}.
\newblock Hyperbolic geometry of complex networks.
\newblock {\em Physical Review E}, 82(3):036106, 2010.

\bibitem{popular}
A.\ Lancichinetti and S.\ Fortunato.
\newblock Limits of modularity maximization in community detection.
\newblock {\em Physical Review E}, 84(6):066122, 2011.

\bibitem{leicht2008community}
E.~A.\ Leicht and M.~E.\ Newman.
\newblock Community structure in directed networks.
\newblock {\em Physical review letters}, 100(11):118703, 2008.

\bibitem{mcdbisect}
M.~J.\ Luczak and C.\ Mc{D}iarmid.
\newblock Bisecting sparse random graphs.
\newblock {\em Random Structures \& Algorithms}, 18(1):31--38, 2001.

\bibitem{majstorovic2014note}
S.\ Majstorovic and D.\ Stevanovic.
\newblock A note on graphs whose largest eigenvalues of the modularity matrix
  equals zero.
\newblock {\em Electronic Journal of Linear Algebra}, 27(1):256, 2014.

\bibitem{bbddiffs}
C.\ McDiarmid.
\newblock On the method of bounded differences.
\newblock In J.\ Siemons, editor, {\em Surveys in Combinatorics}, volume 141 of
  {\em London Mathematical Society Lecture Note Series}, pages 148--188.
  Cambridge {U}niversity {P}ress, 1989.

\bibitem{ERusJournal}
C.\ Mc{D}iarmid and F.\ Skerman.
\newblock Modularity of {Erd{\H{o}}s}-{R}\'enyi random graphs.
\newblock {\em Random Structures \& Algorithms, to appear}.

\bibitem{modERAofA}
C.\ McDiarmid and F.\ Skerman.
\newblock Modularity of {E}rd{\H{o}}s-{R}{\'e}nyi random graphs.
\newblock In {\em 29th International Conference on Probabilistic, Combinatorial
  and Asymptotic Methods for the Analysis of Algorithms}, volume~1, 2018.

\bibitem{treelike}
C.\ Mc{D}iarmid and F.\ Skerman.
\newblock Modularity of regular and treelike graphs.
\newblock {\em Journal of Complex Networks}, 6(4), 2018.

\bibitem{extreme}
C.\ Mc{D}iarmid and F.\ Skerman.
\newblock Extreme values of modularity, in preparation.
\newblock 2019.

\bibitem{sampling}
C.\ Mc{D}iarmid and F.\ Skerman.
\newblock Modularity and edge-sampling, in preparation.
\newblock 2019.

\bibitem{vdense}
C.\ Mc{D}iarmid and F.\ Skerman.
\newblock Modularity of very dense graphs, in preparation.
\newblock 2019.

\bibitem{mckay2016degree}
B.~D.\ McKay and F.\ Skerman.
\newblock Degree sequences of random digraphs and bipartite graphs.
\newblock {\em Journal of Combinatorics}, 7(1):21--49, 2016.

\bibitem{meeks2019parameterised}
K.\ Meeks and F.\ Skerman.
\newblock The parameterised complexity of computing the maximum modularity of a
  graph.
\newblock In {\em 13th International Symposium on Parameterized and Exact
  Computation (IPEC 2018)}. Schloss Dagstuhl-Leibniz-Zentrum fuer Informatik,
  2019.

\bibitem{distinguish}
E.\ Mossel, J.\ Neeman, and A.\ Sly.
\newblock Reconstruction and estimation in the planted partition model.
\newblock {\em Probability Theory and Related Fields}, 162(3-4):431--461, 2015.

\bibitem{NewmanBook}
M.~E.~J.\ Newman.
\newblock {\em Networks: {A}n {I}ntroduction}.
\newblock Oxford University Press, 2010.

\bibitem{NewmanGirvan}
M.~E.~J.\ {N}ewman and M.\ Girvan.
\newblock Finding and evaluating community structure in networks.
\newblock {\em Physical Review E}, 69(2):026113, 2004.

\bibitem{porter2007community}
M.~A.\ Porter, P.~J.\ Mucha, M.~E.\ Newman, and A.~J.\ Friend.
\newblock Community structure in the united states house of representatives.
\newblock {\em Physica A: Statistical Mechanics and its Applications},
  386(1):414--438, 2007.

\bibitem{porter2009communities}
M.~A.\ Porter, J.-P.\ Onnela, and P.~J.\ Mucha.
\newblock Communities in networks.
\newblock {\em Notices of the AMS}, 56(9):1082--1097, 2009.

\bibitem{pralat}
L.~O.\ Prokhorenkova, P.\ Pra{\l}at, and A.\ Raigorodskii.
\newblock Modularity in several random graph models.
\newblock {\em Electronic Notes in Discrete Mathematics}, 61:947--953, 2017.

\bibitem{trulymodular}
J.\ Reichardt and S.\ Bornholdt.
\newblock When are networks truly modular?
\newblock {\em Physica D: Nonlinear Phenomena}, 224(1):20--26, 2006.

\bibitem{thesis}
F.\ Skerman.
\newblock {\em Modularity of Networks}.
\newblock PhD thesis, University of Oxford, 2016.

\bibitem{van2010spectral}
P.\ Van~Mieghem, X.\ Ge, P.\ Schumm, S.\ Trajanovski, and H.\ Wang.
\newblock Spectral graph analysis of modularity and assortativity.
\newblock {\em Physical Review E}, 82(5):056113, 2010.

\end{thebibliography}


\section*{Appendices}
\appendix
\section{Proofs for $\q(G_{n,m})$}\label{sec.Gnmproofs}

In this appendix we use the robustness lemma, Lemma~\ref{lem.edgeSim}, to deduce Propositions~\ref{prop.thm1.1},~\ref{prop.thm1.2} and~\ref{prop.thm1.3} from
Theorems~\ref{thm.usER},~\ref{thm.sparse} (and its proof) and~\ref{thm.growthRate} respectively. We start with an elementary lemma on the binomial distribution $\Bin(n,p)$.
\begin{lemma} \label{lem.bin}
If $0< \eps \leq 1$, $X \sim \Bin(n,p)$, $\sigma^2=np(1-p)$ and $\eps \sigma \geq 1$, then $\pr(|X-np| \leq \eps \sigma) \geq \eps/8$. 
\end{lemma}

\begin{proof}
If $k \geq np$ then $\pr(X=k) \geq \pr(X=k+1)$, and if $k \leq np$ then $\pr(X=k) \geq \pr(X=k-1)$.
Also, for any $x>0$, $\big| \Z \cap (np,np+x]\big|$ is either $\lfloor x \rfloor$ or $\lceil x \rceil$.  Hence
\begin{equation*} \frac{\pr(np<X \leq np + \eps \sigma)}{\pr(np<X \leq np + 2 \sigma)} \geq
\frac{ \big| \Z \cap (np,np+ \eps \sigma]\big|}{\big| \Z \cap (np,np+ 2 \sigma]\big|} \geq \frac{\lfloor \eps \sigma \rfloor}{\lceil 2 \sigma \rceil}.\end{equation*}
But $\lfloor \eps \sigma \rfloor > \eps \sigma /2$ since $\eps \sigma \geq 1$, and 
$\lceil 2 \sigma \rceil < 2 \sigma +1 \leq 3 \sigma$ since $\sigma \geq 1$.  Hence
\begin{equation*} \pr(np<X \leq np + \eps \sigma) \geq (\eps/6) \, \pr(np<X \leq np + 2 \sigma).\end{equation*}
Similarly
\begin{equation*} \pr(np - \eps \sigma \leq X < np) \geq (\eps/6) \, \pr(np - 2 \sigma \leq X < np).\end{equation*}
Adding the last two inequalities we find
\begin{equation*} \pr(0<|X-np| \leq \eps \sigma) \geq (\eps/6) \, \pr(0<|X-np| \leq 2 \sigma) \end{equation*}
and since $\eps/6 \leq 1$ it follows that 
\begin{equation*} \pr(|X-np| \leq \eps \sigma) \geq (\eps/6) \, \pr(|X-np| \leq 2 \sigma). \end{equation*}
But $\pr(|X-np| \leq 2 \sigma) \geq 1-(1/4)$ by Chebyshev's inequality, so
\begin{equation*} \pr(|X-np| \leq \eps \sigma) \geq (\eps/6) (3/4) = \eps/8,\end{equation*}
as required.
\end{proof}
The following lemma will immediately yield Propositions~\ref{prop.thm1.1} and~\ref{prop.thm1.3} from Theorems~\ref{thm.usER} and~\ref{thm.growthRate} respectively.  

\begin{lemma} \label{lem.mtop}
  Let $m=m(n) \to \infty$, let $N= \binom{n}{2}$ and let $p=m/N$.  Suppose that $\q(\Gnp) \in (a_n,b_n)$ whp.  Let $\eps>0$, and let
\begin{equation*}   x_n = \pr\left( \q(G_{n,m}) \not\in (a_n- \eps/\sqrt{m}, b_n + \eps/\sqrt{m}) \right).\end{equation*}
 Then $x_n=o(1)$.
\end{lemma}

\begin{proof}
We can couple $G_{n,m}$ and $G_{n,m'}$ so that, if say $m \geq m'$ then $E(G_{n,m}) \supseteq E(G_{n,m'})$, and so always $|\q(G_{n,m})-\q(G_{n,m'})| \leq 2|m-m'|/m$ by Lemma~\ref{lem.edgeSim}.
Thus, if $|m-m'| \leq (\eps/2)  \sqrt{m}$, then
\begin{equation*} x_n \leq \pr\big( \q(G_{n,m'}) \not\in (a_n, b_n) \big).\end{equation*}
Hence
\begin{eqnarray*}
\pr(\q(\Gnp) \not\in (a_n, b_n))
& = &
\sum_{m'} \pr(e(\Gnp)=m') \, \pr(\q(G_{n,m'}) \not\in (a_n,b_n))\\
& \geq &
\pr\big((|e(\Gnp)-m|  \leq (\eps/2) \sqrt{m} \, \big) \cdot x_n\\
& \geq &
(\eps/16) \cdot x_n
\end{eqnarray*}
by Lemma~\ref{lem.bin}.  It follows that
\begin{equation*} x_n \leq (16/\eps) \, \pr(\q(\Gnp) \not\in (a_n, b_n)) =o(1),\end{equation*}
as required.
\end{proof}

It remains to prove Proposition~\ref{prop.thm1.2}.
Parts (i) and (ii) may be proved much as were the corresponding parts of Theorem~\ref{thm.sparse}.
For part (iii), observe that, from the last inequality in the proof of Lemma~\ref{lem.subER}, we have $q_{\CC}(\Gnp) >1- (4\eps)^2/(1-\eps/4)$ whp; and part (iii) now follows using Lemma~\ref{lem.mtop}.

\newpage
\section{List of modularity values}
For reference we compile a list of some known (maximum) modularity values, sorted into classes with modularity near~1, near or exactly~0, and those with modularity bounded strictly between. Knowing the modularity for classes of graphs may help us to understand the behaviour of the modularity function. We mention results on $\q(\Gnp)$: there are similar results for $\q(G_{n,m})$. We use $n$ for the number of vertices and $m$ for the number of edges.\\
\begin{figure}[h!]
\footnotesize
\thisfloatpagestyle{empty}
$
\renewcommand{\arraystretch}{1.2}
\begin{array}{|l | l | l |  rl | l}
\cline{1-1}\mbox{\textbf{Maximally Modular}}\\ 
\cline{1-5} \mbox{Cycle} &C_n& \q(C_n)=1-2{n}^{-1/2}(1+o(1))&&\mbox{\cite{nphard}[Thm~6.7]} \\ 
\mbox{Tree}  	&T_m \; : \; \Delta(T_m)=o(n) & \displaystyle \q(T_m) \geq 1- 2(2\Delta/m)^{1/2} &&\mbox{\cite{treelike}[Thm 11]} \\
\mbox{Tree-like, i.e. low treewidth}& G_m \; : \; \Delta(G_m) {\rm tw}(G_m)=o(n) &  \displaystyle \q(T_m) \geq 1- 2(({\rm tw}+1)\Delta/m)^{1/2} &&\mbox{\cite{treelike}[Thm 11]}  \\
&&&&&\\
\renewcommand{\arraystretch}{1.2}
\mbox{(whp) }\mbox{Erd\H{o}s-R\'enyi}
&\mbox{if } n^2p\rightarrow \infty\; \& \;np\leq 1+o(1) &\mbox{(whp) }\q(\Gnp)=1+o(1) &&\mbox{Thm~\ref{thm.usER}(a)} \\
&&&&\\
\mbox{(whp) }\mbox{Random Planar}&G_n&\mbox{(whp) } \q(G_n) = 1 - O(\frac{\log n}{\sqrt {n}}) &&\mbox{\cite{treelike}[Cor 12]}\\
&&&&\\
\mbox{(whp) Random 2-regular} &G_{n,2} &\mbox{(whp) } \q(G_{n,2})=1-\frac{2}{\sqrt{n}}+o(\frac{\log^2 n}{n}) &&\mbox{\cite{treelike}[Prop 2]}   \\
&&&&\\
&&&&\\
\mbox{\textbf{Critically Modular}}&&&&\\ 
\cline{1-5} \mbox{(whp) Erd\H{o}s-R\'enyi} 
&\mbox{$\exists b$, if $np=c >1$ then}&\mbox{(whp) } \frac{0.668}{\sqrt{c}} <\q(\Gnp) < \frac{b}{\sqrt{c}} &&\mbox{Thms~\ref{thm.growthRate}, \ref{thm.lowerSqrt2} } \\
& & &&   \\
\mbox{(whp) Random cubic} &G_{n,3} &\mbox{(whp) } 0.66<\q(G_{n,3})<0.81 &&\mbox{\cite{treelike}[Thm 6]}   \\
&&&&\\
\mbox{(whp) Random $r$-regular} & G_{n,r} \mbox{ for }r=4,\ldots,12  &\mbox{see paper}&&\mbox{\cite{treelike}[Thm 6]}  \\
& G_{n,r} \mbox{ for fixed }r\geq r_0  & \mbox{(whp) } \frac{0.76}{\sqrt{r}}<\q(G_{n,r})<\frac{2}{\sqrt{r}} && \mbox{\cite{treelike,pralat} } 
\\
&&&&\\
\mbox{(whp) Preferential attachment} & \mbox{$h\geq 2$ edges added per step}& \mbox{(whp) } \frac{1}{h}<\q(G_{n}^h)<0.94  && \mbox{\cite{pralat}[Thm~10]} \\
&&&&\\
&&&&\\
\mbox{\textbf{Minimally Modular}} &&&&\\ \cline{1-5}
\mbox{(whp) Erd\H{o}s-R\'enyi}& 
\mbox{if }np\rightarrow \infty & \mbox{(whp) } \q(\Gnp)=o(1) && \mbox{Thm~\ref{thm.usER}(c)} \\ 
&&&&\\
&&&&\\
\mbox{\textbf{Non-Modular}} &&&&\\ \cline{1-5}
\mbox{Complete}& K_n  & \q(K_n)=0 && \mbox{\cite{nphard}[Thm~6.3]}\\ 
\mbox{Complete multipartite}& K_{n_1,\ldots,n_k} & \q(K_{n_1,\ldots,n_k})=0 &&\mbox{\cite{majstorovic2014note,bolla2015spectral}}\\
\mbox{Nearly complete}& \mbox{for $G$ with }m\geq\binom{n}{2}-n/2 & \q(G)=0 && \mbox{\cite{vdense}}\\  
&&&&\\
\mbox{(whp) Erd\H{o}s-R\'enyi}& \mbox{if }p\geq 1-c/n, c<1 & \mbox{(whp) } \q(\Gnp)=0 && \mbox{\cite{vdense}} \\  \cline{1-5}
\end{array}
$
\label{fig.modZoo}
\end{figure}

\end{document}